\documentclass[a4paper]{amsart}
\pdfoutput=1

\usepackage[all,hyperref]{paper_diening}
\renewcommand{\norm}[1]{\left\lVert#1\right\rVert}
\renewcommand{\abs}[1]{\left\lvert#1\right\rvert}

\usepackage[backend=biber,maxbibnames=5,sorting=nyt,maxalphanames=5,style=alphabetic,bibencoding=utf8,giveninits,url=false,isbn=false]{biblatex}
\AtBeginBibliography{\small}

\usepackage{amsmath,amsthm,amssymb}

\usepackage{epsfig}
\usepackage[utf8]{inputenc}
\usepackage[T1]{fontenc}
\usepackage{dsfont}
\usepackage{faktor}

\usepackage{tikz}
\usepackage{pgfplots}
\pgfplotsset{
  width=.65\linewidth,
  axis background/.style={fill=black!5!white},
  grid style={densely dotted,semithick},
  legend style={
    legend columns=1,
    legend pos=outer north east
  },
  compat=newest 
}
 
\pgfplotscreateplotcyclelist{MyColors}{%
    {red,mark = *,every mark/.append style={solid,scale=0.4,fill=red}},
    {teal,mark = x,every mark/.append style={solid,scale=0.4,fill=teal}},
    {blue,mark = square*,every mark/.append style={solid,scale=0.4,fill=blue}},
{dotted,red,mark = *,every mark/.append style={solid,scale=0.4,fill=red}},
    {dotted,teal,mark = x,every mark/.append style={solid,scale=0.4,fill=teal}},
    {dotted,blue,mark = square*,every mark/.append style={solid,scale=0.4,fill=blue}},
{dash dot,red,mark = *,every mark/.append style={solid,scale=0.4,fill=red}},
    {dash dot,teal,mark = x,every mark/.append style={solid,scale=0.4,fill=teal}},
    {dash dot,blue,mark = square*,every mark/.append style={solid,scale=0.4,fill=blue}},}

\pgfplotscreateplotcyclelist{MyColors4}{%
    {red,mark = *,every mark/.append style={solid,scale=0.4,fill=red}},
    {teal,mark = x,every mark/.append style={solid,scale=0.4,fill=teal}},
    {blue,mark = square*,every mark/.append style={solid,scale=0.4,fill=blue}},
    {green,mark = o,every mark/.append style={solid,scale=0.4,fill=green}},
{dotted,red,mark = *,every mark/.append style={solid,scale=0.4,fill=red}},
    {dotted,teal,mark = x,every mark/.append style={solid,scale=0.4,fill=teal}},
    {dotted,blue,mark = square*,every mark/.append style={solid,scale=0.4,fill=blue}},
    {dotted,green,mark = o,every mark/.append style={solid,scale=0.4,fill=green}},
{dash dot,red,mark = *,every mark/.append style={solid,scale=0.4,fill=red}},
    {dash dot,teal,mark = x,every mark/.append style={solid,scale=0.4,fill=teal}},
    {dash dot,blue,mark = square*,every mark/.append style={solid,scale=0.4,fill=blue}},}

\usepackage{mathrsfs}

\numberwithin{equation}{section}

\newcommand{\Div}{\divergence}
\newcommand{\R}{\mathbb R}

\newcommand{\dd}{\,\mathrm{d}}
\newcommand{\ds}{\dd s}
\newcommand{\dt}{\dd t}
\newcommand{\dx}{\dd x}

\providecommand{\seminormtmp}[2]{{#1[{#2}#1]}}
\providecommand{\seminorm}[1]{\seminormtmp{}{#1}}

\begin{document}

\title[Averaged discretization]{An averaged space-time discretization of the stochastic $p$-Laplace system}

\author{ Lars Diening, Martina Hofmanov\'{a} and J\"{o}rn Wichmann}%
\address[L. Diening, M. Hofmanov\'{a},
J. Wichmann]{Department of Mathematics, University of Bielefeld,
  Postfach 10 01 31, 33501 Bielefeld, Germany}%
\email{lars.diening@uni-bielefeld.de}%
\email{hofmanova@math.uni-bielefeld.de}%
\email{jwichmann@math.uni-bielefeld.de}%
\thanks{ The research was funded by the Deutsche Forschungsgemeinschaft (DFG, German Research Foundation) – SFB 1283/2 2021 – 317210226.}

\begin{abstract}
  We study the stochastic $p$-Laplace system in a bounded domain. We propose two new space-time discretizations based on the approximation of time-averaged values. We establish linear convergence in space and $1/2$ convergence in time. Additionally, we provide a sampling algorithm to construct the necessary random input in an efficient way. The theoretical error analysis is complemented by numerical experiments.
\end{abstract}

\subjclass[2010]{%
  65N15, 
  65N30, 
  35K55, 
  35K65, 
  35K67, 
  60H15, 
}
\keywords{Stochastic PDEs, Nonlinear Laplace-type systems, Finite
  element methods, Space-time discretization, stochastic $p$-heat equation}

\maketitle

\tableofcontents

\section{Introduction}

Let $\mathcal{O} \subset \R^n$ be a polygonal domain, $n \geq 1$, $N \geq 1$, $T > 0$ be finite. We are interested in the approximation of the solution process $u :\Omega \times [0,T] \times \mathcal{O} \to \R^N$ to the stochastic $p$-Laplace system. Given an initial datum $u_0$ and a stochastic forcing term $G(u) \dd W$ (for the precise assumptions see Assumption \ref{ass:Noise}), $u$ is determined by the relations 
\begin{alignat}{2} \label{intro:p-Laplace-stoch}
\begin{aligned}
\dd u - \Div S(\nabla u) \dt  &= G(u) \dd W \quad &&\text{ in } \Omega \times (0,T) \times \mathcal{O}, \\
u &= 0 \quad &&\text{ on } \Omega \times (0,T) \times \partial \mathcal{O}, \\
u(0) &= u_0 &&\text{ on } \Omega \times \mathcal{O},
\end{aligned}
\end{alignat}
where $S(\xi):= \left( \kappa + \abs{\xi} \right)^{p-2} \xi \in \R^{n \times N}$, $p \in (1,\infty)$ and $\kappa \geq 0$. Closely related to $S$ is the nonlinear operator $V(\xi):= \left( \kappa + \abs{\xi} \right)^\frac{p-2}{2} \xi \in \R^{n \times N}$.

\subsection{Model}
The system~\eqref{intro:p-Laplace-stoch} has many important applications in nature. As one major example, it provides a prototype system towards the modeling of non-Newtonian fluids. More specifically, it is closely related to power law fluids \cite{Breit2015, MR2884052}. The case~$p=2$ corresponds to the famous stochastic heat equation and has been studied extensively, both analytically e.g. \cite{MR1316109, MR3779690, MR3354615} as well as numerically e.g. \cite{MR2476562, MR2507605, MR3454367,MR3543162,MR4268037}. 

The $p$-Laplace system arises as a stochastically perturbed gradient flow of the energy $J : W^{1,p}_0(\mathcal{O}) \to [0,\infty)$ defined by
\begin{align} \label{intro:p-Energy}
\mathcal{J}(u) := \int_\mathcal{O} \varphi(\abs{\nabla u}) \dx ,
\end{align}
where $\varphi(t):= \int_0^t (\kappa+s)^{p-2}s \ds$. In the case $\kappa = 0$ the energy \eqref{intro:p-Energy} corresponds to the classical $p$-Dirichlet energy. 

\subsection{Existence \& well-posedness}
The existence of analytically weak solutions in the space $ L^2\left(\Omega; C\left([0,T]; L^2\left(\mathcal{O} \right)\right)\right) \cap L^p\big(\Omega;L^p\big(0,T;W^{1,p}_0\left(\mathcal{O}\right)\big)\big)$
can be established by standard monotonicity arguments \cite{LiRo}. It requires a linear growth assumption on the noise coefficient $G$ and $L^2(\mathcal{O})$-integrable initial data. 

First results on the existence of strong solutions to stochastically perturbed gradient flows have been obtained by Gess \cite{Gess2012Strong}. It includes the degenerate, $p \geq 2$, $p$-Laplace equation. 

In the literature different generalization of~\eqref{intro:p-Laplace-stoch} have been considered. Well-posedness for merely $L^1$-initial data has been addressed in \cite{MR4225916}. More general systems, where $p$ is allowed to depend on $(\omega,t,x)$ respectively on $(t,x)$, are considered in \cite{MR3585706}, \cite{MR3535765} respectively in \cite{MR3327516}. The singular case $p \in [1,2)$ has been analyzed in \cite{MR2561270}, \cite{MR3580815} and \cite{MR4261330}.

\subsection{Regularity}
Regularity properties of a function $u$ become particularly important when it comes to the approximability of $u$ within a discrete function class. Prominently, discrete tensor spaces generated via a time stepping scheme and a finite element discretization in space can approximate smooth functions more easily compared to non-smooth functions.

Historically, many authors addressed H\"older and $C^{1,\alpha}$-regularity of solutions to the deterministic stationary $p$-Laplace equation, e.g. \cite{MR0244628,MR474389,MR709038,MR727034,MR969499,MR2062943,MR2086757}. 
In general $\alpha \in (0,1)$ is an unknown quantity. While $C^{1,\alpha}$-regularity can be used to measure the approximation quality of finite elements, they usually fail to produce optimal results, since in general $\alpha < 1$. 

Sobolev regularity provides an alternative scale of smoothness. In the singular case, $p \in (1,2)$, $W^{2,2}$-regularity has been proven by Liu and Barrett \cite{LiuBar93rem}. The degenerate setting is more delicate and one can not expect $\nabla^2 u$ to be well-defined on $\set{\nabla u = 0}$. In fact, due to the nonlinear structure of the equation, solutions have limited regularity even for smooth data. A sharp result in the $2$-dimensional setting about limited regularity for the H\"older as well as Sobolev scale has been obtained by Iwaniec and Manfredi \cite{IwMa}.

The nonlinear character of the equation naturally introduces the additional quantities $S(\nabla u)$ and $V(\nabla u)$. In the scalar case both expressions are robust with respect to $p \in (1,\infty)$. Here it is possible to prove $V(\nabla u) \in W^{1,2}$ via a difference quotient technique \cite{MR1055157, MR1237129} and $S(\nabla u) \in W^{1,2}$ by a functional inequality \cite{MR3803772}. The result $V(\nabla u) \in W^{1,2}$ generalizes to the vectorial case. However, the functional inequality fails in the vectorial case at least point-wise for $p \leq 2(2-\sqrt{2})\approx 1.1715$. Therefore, it is unclear whether a regularity result $S(\nabla u) \in W^{1,2}$ is achievable for small $p$. Nevertheless, for $p > 2(2-\sqrt{2})$ it is shown in \cite{Balci2021} that $S(\nabla u) \in W^{1,2}$. Regularity for $S(\nabla u)$ on the Besov and Triebel-Lizorkin scale in the plane for $p \geq 2$ has been obtained in \cite{MR4021898}.

Regularity results for the parabolic $p$-Laplace system were derived in \cite{DER}. The authors showed, by formally testing the equation with $-\Delta u$ respectively $-\partial_t^2 u$, that
\begin{subequations}
\label{intro:instat}
\begin{align}
V(\nabla u) &\in L^2\left(0,T;W^{1,2}\left(\mathcal{O}\right)\right) \cap W^{1,2}\left(0,T;L^2\left(\mathcal{O}\right)\right), \label{intro:V-instat} \\
u &\in  L^\infty\left(0,T;W^{1,2}\left(\mathcal{O}\right)\right) \cap W^{1,\infty}\left(0,T;L^2\left(\mathcal{O}\right)\right). \label{intro:u-instat}
\end{align}
\end{subequations}
Additionally, $S(\nabla u) \in L^2(0,T;W^{1,2}(\mathcal{O}))$ was proven in \cite{CiaMaz20} for either $p\in(1,\infty)$ and $N=1$ or $p \geq 2$ and $N > 1$.

Within the context of the stochastic $p$-Laplace system it is possible to prove similar spatial regularity as in~\eqref{intro:instat}. The formal testing needs to be replaced by a suitable application of It\^o's formula as done by Breit \cite{Breit2015Regularity}. However, the time regularity is limited due to the presence of the stochastic forcing. In the degenerate case partial time regularity can be recovered by exploiting the strong formulation of the $p$-Laplace system as done by Wichmann \cite{wichmann2021temporal}. Overall, for appropriate data assumptions it is possible to verify (see Section~\ref{sec:Regularity})
\begin{subequations}
\label{intro:est-stoch}
\begin{align}
V(\nabla u) &\in L^2\left(\Omega;L^2\left(0,T;W^{1,2}\left(\mathcal{O}\right)\right) \cap B^{1/2}_{2,\infty}\left(0,T;L^2\left(\mathcal{O}\right)\right)\right),\label{intro:Vest-stoch}\\
u &\in L^2\left(\Omega; L^\infty\left(0,T;W^{1,2}\left(\mathcal{O}\right)\right) \cap B^{1/2}_{\Phi_2,\infty}\left(0,T;L^2\left(\mathcal{O}\right)\right)\right). \label{intro:uest-stoch}
\end{align}
\end{subequations}
Here $\Phi_2(s) := e^{s^2}-1$ and $B$ denotes a Besov-Orlicz space (see Section~\ref{sec:Function spaces}).

\subsection{Approximation}
In the past many authors have studied the numerical approximation of the deterministic counterpart of \eqref{intro:p-Laplace-stoch}, e.g.  \cite{Wei1992,BaLi1,BaLi2,EbLi,DER,DieKre08,BelDieKre2012,
BarDieNoc20,BreMen18pre,MR4286257,MR4367655}. 

The error of the discrete and analytic solution has been expressed in various different quantities. It turned out that the natural quantity to measure the error is given by 
\begin{align}\label{intro:error-det}
\max_{0\leq m \leq M} \norm{u(t_m) - u_{m,h}}_{L^2(\mathcal{O})}^2 + \tau \sum_{m=1}^M \norm{V(\nabla u(t_m)) - V(\nabla u_{m,h})}_{L^2(\mathcal{O})}^2,
\end{align}
where $V(\xi) := \left( \kappa+ \abs{\xi} \right)^\frac{p-2}{2}\xi$ and $u_{m,h}$ is an approximation of $u(t_m)$. In \cite{DieKre08} it has been proved that the expression $\norm{V(\nabla u) - V(\nabla u_h)}_{L^2(\mathcal{O})}^2$ is equivalent to the energy error $\mathcal{J}(u_h) - \mathcal{J}(u)$. 

In the stationary setting, starting with the seminal work by Barrett and Liu \cite{BaLi1} and further improvements in \cite{EbLi} and \cite{DR}, it has been proved that 
\begin{align*}
\norm{V(\nabla u) - V(\nabla u_h)}_{L^2(\mathcal{O})} \lesssim h \norm{\nabla V(\nabla u)}_{L^2(\mathcal{O})}.
\end{align*}
This settles the question about optimal convergence for piece-wise linear continuous elements. In fact, the paper \cite{DER} deals with the parabolic system and optimal convergence under the regularity assumption~\eqref{intro:instat} has been achieved, i.e., 
\begin{align} \label{intro:est-instat}
\begin{aligned}
\max_{0\leq m \leq M} \norm{u(t_m) - u_{m,h}}_{L^2(\mathcal{O})}^2 &+ \tau \sum_{m=1}^M \norm{V(\nabla u(t_m)) - V(\nabla u_{m,h})}_{L^2(\mathcal{O})}^2 \\
&\lesssim h^2 + \tau^2.
\end{aligned}
\end{align}

The error quantity \eqref{intro:est-instat} relies on the fact that the mapping $t \mapsto \norm{V(\nabla u(t))}_{L^2(\mathcal{O})}$ is continuous. However, if the data is not sufficiently smooth, point-values might not be well-defined. In general, dealing with irregular data and therefore irregular solutions is a delicate task. Different methods have been suggested to recover well-defined error quantities.

In \cite{MR4286257} two of us develop a numerical scheme based on time averages to circumvent the usage of point evaluations. A more probabilistic approach has been used in \cite{MR4029845}. There the authors replace deterministic evaluation points by random ones. 

First results for monotone stochastic equations were derived by Gy\"ongy and Millet \cite{MR2139212,MR2465711}. They developed an abstract discretization theory that also covers the system~\eqref{intro:p-Laplace-stoch}. However, they only deal with the degenerate case and require restrictive assumptions on the regularity of the solution. 

In the stochastic case, due to the limited time regularity~\eqref{intro:est-stoch}, robust error quantities need to be used. Breit, Hofmanov\'a and Loisel~\cite{MR4298537} use randomized time steps to construct an algorithm that achieves almost optimal convergence in time and optimal convergence in space, i.e., for all $\alpha \in (0,1/2)$,
\begin{align*}
&\mathbb{E}_{\bft} \otimes \mathbb{E} \left[ \max_{0\leq m \leq M} \norm{u(\bft_m) - u_{m,h}}_{L^2(\mathcal{O})}^2 \hspace{-1pt}+ \hspace{-1pt}\sum_{m=1}^M \hspace{-1pt} \int_{\bft_{m-1}}^{\bft_{m}} \hspace{-3pt} \norm{V(\nabla u(s)) - V(\nabla u_{m,h})}_{L^2(\mathcal{O})}^2 \ds \right] \\
&\hspace{6em}\lesssim h^2 + \tau^{2\alpha}.
\end{align*} 
Here $\mathbb{E}_{\bft}$ denotes the expectation with respect to the random time steps. They use H\"older and Sobolev-Slobodeckij spaces to measure the time regularity of the solution. This excludes the limiting case $\alpha = 1/2$.

\subsection{Main results}
Based on deterministic time averages we propose two algorithms~\eqref{eq:algo} and~\eqref{algo:2nd} that are essentially driven by the update rule, $\mathbb{P}$-a.s. for all $m \geq 2$ and $\xi_h \in V_h$
\begin{align} \label{intro:algo}
 \left(v_m - v_{m-1}, \xi_h \right) + \tau \left( S(\nabla v_m), \nabla \xi_h \right) = \left( G(v_{m-2})\Delta_m \mathbb{W}, \xi_h \right).
\end{align}
The randomness enters through the averaged increments $\Delta_m \mathbb{W}:= \mean{W}_m - \mean{W}_{m-1}$, where $\mean{W}_m$ is the time averaged value of $W$ on the interval $[t_{m-1},t_m]$. 

Importantly, we manage to achieve optimal convergence in time with rate $1/2$ without assuming any time H\"older regularity on the solution process. Instead, we measure time regularity in terms of Nikolskii spaces. Our main results, Theorem~\ref{thm:Convergence} and~\ref{thm:Convergence2}, verify under the condition
\begin{subequations}\label{intro:reg-sol}
\begin{align} 
 u &\in L^2\left( \Omega; B^{1/2}_{2,\infty}\left(0,T; L^2\left(\mathcal{O} \right) \right)  \cap L^\infty\left(0,T; W^{1,2}\left(\mathcal{O}\right)\right)\right), \\
 V(\nabla u) &\in L^2\left( \Omega; B^{1/2}_{2,\infty}\left(0,T; L^2\left( \mathcal{O} \right) \right) \cap  L^2 \left( 0,T; W^{1,2}\left( \mathcal{O} \right) \right) \right),
\end{align}
\end{subequations}
the optimal convergence
\begin{align*}
&\mathbb{E} \left[ \max_{m=1,\ldots,M} \norm{\mean{u}_m - v_m}_{L^2_x}^2 + \sum_{m=1}^M\int_{t_{m-1}}^{t_m} \norm{V(\nabla u(s)) - V(\nabla v_m)}_{L^2_x}^2 \dd s \right] \\
&\hspace{3em}\lesssim h^2   + \tau.
\end{align*}
We want to stress that the regularity assumption~\eqref{intro:reg-sol} is even weaker than the provable regularity~\eqref{intro:est-stoch}. 

On the other hand if the solution process $u$ has certain amount of time regularity, e.g.~\eqref{intro:uest-stoch}, one can recover point evaluations, cf. Lemma~\ref{lem:exponential_stability},
\begin{align*}
 \mathbb{E}\left[ \max_{m\in\set{1,\ldots,M}} \norm{u(t_m) - \mean{u}_m}_{L^2_x}^2 \right] \lesssim \tau \ln(1+\tau^{-1})  \mathbb{E} \left[ \seminorm{u}_{ B^{1/2}_{\Phi_2,\infty} L^2_x}^2 \right] .
\end{align*}

For the implementation of~\eqref{intro:algo} one needs to sample according to the distribution of the random variable $\Delta_m \mathbb{W}$. We show in Corollary~\ref{cor:Distr-IncMean} that $\Delta_m \mathbb{W}$ is a Gaussian random variable whose variance is slightly reduced compared to the classical increments $\Delta_m W:= W(t_m) - W(t_{m-1})$. Ultimately, we provide a sampling algorithm~\eqref{algo:Sampling} that samples not only the marginal distributions but the joint distribution of the random vector $(\Delta_m W, \Delta_m \mathbb{W})_{m=1}^M$.


\subsection{Outline}
In Section~\ref{sec:2} we formulate the functional analytic setup, construct the multiplicative forcing term $G$ and recall known regularity results. Section~\ref{sec:Numerics} introduces the discrete setup and contains the main results Theorem~\ref{thm:Convergence} and Theorem~\ref{thm:Convergence2}. Next, Section~\ref{sec:Discrete} clarifies the construction of the discrete random input vectors and provides a sampling algorithm. Lastly, Section~\ref{sec:Simulations} contains numerical experiments.

\section{Mathematical setup} \label{sec:2}
This section contains classical definitions and preliminary results. It is structured as follows: Section~\ref{sec:Function spaces} introduces the function analytical framework. Section~\ref{sec:StochasticIntegral} presents the construction of the stochastic forcing. Section~\ref{sec:gradient_flow} is about the nonlinear operators $S$ and $V$. The solution concept is fixed in Section~\ref{sec:Solution}. Lastly, Section~\ref{sec:Regularity} collects regularity results. 

 Let $\mathcal{O} \subset \R^n$ for $n \geq 1$ be a bounded Lipschitz domain (further assumptions on $\mathcal{O}$ will be needed for the regularity of solutions). For some given $T>0$ we denote by $I := [0,T]$ the time interval and write $\mathcal{O}_T := I \times \mathcal{O}$ for the time space cylinder. Moreover let $\left(\Omega,\mathcal{F}, (\mathcal{F}_t)_{t\in I}, \mathbb{P} \right)$ denote a stochastic basis, i.e. a probability space with a complete and right continuous filtration $(\mathcal{F}_t)_{t\in I}$. We write $f \lesssim g$ for two non-negative quantities $f$ and $g$ if $f$ is bounded by $g$ up to a multiplicative constant. Accordingly we define $\gtrsim$ and $\eqsim$. We denote by $c$ a generic constant which can change its value from line to line.
\subsection{Function spaces} \label{sec:Function spaces}
As usual, $L^q(\mathcal{O})$ denotes the Lebesgue space and $W^{1,q}(\mathcal{O})$ the Sobolev space, where $1\leq q < \infty$. We denote by $W^{1,q}_0(\mathcal{O})$ the Sobolev spaces with zero boundary values. It is the closure of $C^\infty_0(\mathcal{O})$ (smooth functions with compact support) in the $W^{1,q}(\mathcal{O})$-norm. We denote by $W^{-1,q'}(\mathcal{O})$ the dual of $W^{1,q}_0(\mathcal{O})$. We do not distinguish in the notation between vector- and matrix-valued functions.

For a Banach space $\left(X, \norm{\cdot}_X \right)$ let $L^q(I;X)$ be the Bochner space of Bochner-measurable functions $u: I \to X$ satisfying $t \mapsto \norm{u(t)}_X \in L^q(I)$. Moreover, $C(I;X)$ is the space of continuous functions with respect to the norm-topology. We also use $C^{\alpha}(I;X)$ for the space of H\"older continuous functions. Given an Orlicz-function $\Phi: [0,\infty] \to [0,\infty]$, i.e. a convex function satisfying $ \lim_{t \to 0} \Phi(t)/t = 0$ and $\lim_{t \to \infty} \Phi(t)/t = \infty$ we define the Luxemburg-norm 
\begin{align*}
\norm{u}_{L^\Phi(I;X)} := \inf \left\{ \lambda > 0 : \int_I \Phi \left( \frac{\norm{u}_X}{\lambda} \right) \ds \leq 1 \right\}.
\end{align*}
The Orlicz space $L^\Phi(I;X)$ is the space of all Bochner-measurable functions with finite Luxemburg-norm. For more details on Orlicz-spaces we refer to \cite{DiHaHaRu}. Given $h \in I$ and $u :I \to X$ we define the difference operator $\tau_h: \set{u: I \to X} \to \set{u: I\cap I - \set{h} \to X} $ via $\tau_h(u) (s) := u(s+h) - u(s)$. The Besov-Orlicz space $B^\alpha_{\Phi,r}(I;X)$ with differentiability $\alpha \in (0,1)$, integrability $\Phi$ and fine index $r \in (1,\infty]$ is defined as the space of Bochner-measurable functions with finite Besov-Orlicz norm $\norm{\cdot}_{B^\alpha_{\Phi,r}(I;X)}$, where
\begin{align*}
\norm{u}_{B^\alpha_{\Phi,r}(I;X)} &:= \norm{u}_{L^{\Phi}(I ;X)} + \seminorm{u}_{B^\alpha_{\Phi,r}(I;X)}, \\
\seminorm{u}_{B^\alpha_{\Phi,r}(I;X)} &:= \left( \int_{I} h^{-r\alpha} \norm{\tau_h u}_{L^\Phi(I\cap I - \set{h};X)}^r \dd h \right)^\frac{1}{r}.
\end{align*} 
In the case $r = \infty$ the integral in $h$ is replaced by an essential supremum and the space is commonly called Nikolskii-Orlicz space. When $\Phi(t) = t^p$ for some $p \in (1,\infty)$ we call the space $B^\alpha_{\Phi,r}(I;X) =B^\alpha_{p,r}(I;X)$ Besov space.

Similarly, given a Banach space $\left(Y, \norm{\cdot}_Y \right)$, we define $L^q(\Omega;Y)$ as the Bochner space of Bochner-measurable functions $u: \Omega \to Y$ satisfying $\omega \mapsto \norm{u(\omega)}_Y \in L^q(\Omega)$. The space $L^q_{\mathcal{F}}(\Omega \times I;X)$ denotes the subspace of $X$-valued $(\mathcal{F}_t)_{t\in I}$-progressively measurable processes. Let $\left(U,\norm{\cdot}_U \right)$ be a separable Hilbert space. $L_2(U;L^2(\mathcal{O}))$ denotes the space of Hilbert-Schmidt operators from $U$ to $L^2(\mathcal{O})$ with the norm $\norm{z}_{L_2(U;L^2_x)}^2:= \sum_{j\in \mathbb{N}} \norm{z(u_j)}_{L^2(\mathcal{O})}^2 $ where $\set{u_j}_{j \in \mathbb{N}}$ is some orthonormal basis of $U$. We abbreviate the notation $L^q_\omega L^q_t L^q_x := L^q(\Omega;L^q(I;L^q(\mathcal{O}))) $ and $L^{q-} := \bigcap_{r< q} L^r$.

\subsection{Stochastic integrals} \label{sec:StochasticIntegral}
In order to construct the stochastic forcing term, we impose the following conditions:
\begin{assumption}\label{ass:Noise}
\begin{enumerate}
\item Let $\left( U, \norm{\cdot}_U \right)$ be a separable Hilbert space. We assume that $W$ is an $U$-valued cylindrical Wiener process with respect to $(\mathcal{F}_t)_{t\in I}$. Formally $W$ can be represented as
\begin{align} \label{rep:W}
W = \sum_{j \in \mathbb{N}} u_j \beta^j,
\end{align}
where $\set{u_j}_{j \in \mathbb{N}}$ is an orthonormal basis of $U$ and $\set{\beta^j}_{j\in \mathbb{N}}$ are independent $1$-dimensional standard Brownian motions.
\item \label{ass:cond2} Let $v \in L^2_{\mathcal{F}}(\Omega \times I; L^2_x)$. We assume that $G(v)(\cdot): U \to L_{\mathcal{F}}^2(\Omega \times I; L^2_x)$ is given by 
\begin{align*}
u \mapsto G(v)(u) := \sum_{j \in \mathbb{N}} g_j(\cdot,v)( u_j, u)_U ,
\end{align*}
where $\set{g_j}_{j\in \mathbb{N}} \in C(\mathcal{O} \times \R^N; \R^N)$ with
\begin{enumerate}[label=(\roman*)]
\item (sublinear growth) for all $x \in \mathcal{O}$ and $\xi \in \R^N$ it holds
\begin{align}\label{ass:growth}
\sum_{j\in \mathbb{N}} \abs{g_j(x,\xi)}^2  \leq c_{\text{growth}}(1+\abs{\xi}^2),
\end{align} 
\item (Lipschitz continuity) for all $x \in \mathcal{O}$ and $\xi_1, \xi_2 \in \R^N$ it holds
\begin{align} \label{ass:Lipschitz}
\sum_{j\in \mathbb{N}} \abs{g_j(x,\xi_1) -g_j(x,\xi_2) }^2 \leq c_{\text{lip}} \abs{\xi_1 - \xi_2}^2.
\end{align} 
\end{enumerate}
\end{enumerate}
\end{assumption}

Now it is a classical result, see for example the book of Pr\'{e}v\^{o}t and R\"ockner \cite{MR2329435}, that we can construct a corresponding stochastic integral.
\begin{proposition}\label{prop:ConstructionStochIntegral}
Let Assumption \ref{ass:Noise} be true. Then the operator $\mathcal{I}$ defined through
\begin{align}\label{def:stoch_integral}
\mathcal{I}(G(v)) := \int_0^{\bullet} G(v)(\dd W_s) := \sum_{j \in \mathbb{N}} \int_0^{\bullet} g_j(\cdot,v) \dd \beta^j_s 
\end{align} 
defines a bounded linear operator from $L^2_{\mathcal{F}}(\Omega \times I;L_2(U; L^2_x))$ to $ L^2_\omega C_t L^2_x $. Moreover,
\begin{enumerate}
\item $\mathcal{I}(G(v))$ is an $L^2_x$-valued martingale with respect to $\left(\mathcal{F}_t\right)_{t \in I}$, 
\item (It\^o isometry) for all $t \in I$ it holds
\begin{align} \label{eq:ito-iso}
\mathbb{E} \left[ \norm{\mathcal{I}(G(v)) (t)}_{L^2_x}^2 \right] =\mathbb{E} \left[\int_0^t \norm{G(v)}^2_{L_2(U;L^2_x)} \ds \right] .
\end{align}
\end{enumerate}
\end{proposition}

\subsection{Perturbed gradient flow} \label{sec:gradient_flow}
Let $\kappa \geq 0$ and $p \in (1,\infty)$. For $\xi \in \R^{n\times N}$ we define
\begin{align} \label{eq:S}
S(\xi) := \varphi'(\abs{\xi}) \frac{\xi}{\abs{\xi}} = \left(\kappa + \abs{\xi} \right)^{p-2} \xi
\end{align}
and
\begin{align}\label{eq:V}
V(\xi) := \sqrt{\varphi'(\abs{\xi})} \frac{\xi}{\abs{\xi}} = \left( \kappa + \abs{\xi} \right)^\frac{p-2}{2} \xi,
\end{align}
where $\varphi(t):= \int_0^t \left(\kappa + s \right)^{p-2} s \ds$. The nonlinear functions $S$ and $V$ are closely related. In particular the following lemmata are of great importance. The proofs can be found in e.g. \cite{DieE08}. For more details we refer to \cite{DR,BelDieKre2012,DieForTomWan20}.
\begin{lemma}[$V$-coercivity] \label{lem:V-coercive}
Let $\xi_1, \xi_2 \in \R^{n\times N}$. Then it holds
\begin{align} \label{eq:V-coercive}
\begin{aligned}
(S(\xi_1) - S(\xi_2)) : (\xi_1 - \xi_2) &\eqsim \abs{V(\xi_1) - V(\xi_2)}^2 \\
&\eqsim \left( \kappa + \abs{\xi_1} + \abs{\xi_1 - \xi_2} \right)^{p-2} \abs{\xi_1 - \xi_2}^2.
\end{aligned}
\end{align}
\end{lemma}
\begin{lemma}[generalized Young's inequality] \label{lem:gen-young}
Let $\xi_1, \xi_2, \xi_3 \in \R^{n\times N}$ and $\delta >0$. Then there exists $c_\delta \geq 1$ such that
\begin{align}\label{eq:gen-young}
\left( S(\xi_1) - S(\xi_2) \right): \left(\xi_2 - \xi_3 \right) \leq \delta \abs{V(\xi_1) - V(\xi_2)}^2 + c_\delta \abs{V(\xi_2) - V(\xi_3)}^2.
\end{align} 
\end{lemma}
\begin{lemma} \label{lem:1side-young}Let $\xi_1, \xi_2, \xi_3 \in \R^{n\times N}$ and $\delta >0$. Then there exists $c_\delta \geq 1$ such that
\begin{align} \label{eq:1side-young}
\left( S(\xi_1) - S(\xi_2) \right) : \xi_3 \leq \delta \abs{V(\xi_1) - V(\xi_2)}^2 \hspace{-2pt}+ c_\delta \left(\kappa + \abs{\xi_1} + \abs{\xi_1 - \xi_2}\right)^{p-2}\abs{\xi_3}^2.
\end{align}
\end{lemma}
\begin{remark}
Lemma \ref{lem:V-coercive} and \ref{lem:gen-young} are still valid if one replaces $\varphi$ in \eqref{eq:S} and \eqref{eq:V} by any uniformly convex $N$-function (cf. \cite{DieForTomWan20}).
\end{remark} 
Given some initial condition $u_0 : \Omega \times \mathcal{O} \to \R^N$ and a stochastic force $(G,W)$ in the sense of Assumption \ref{ass:Noise} we are interested in the system
\begin{subequations}
\label{eq:p-Laplace}
\begin{alignat}{2}\label{eq:Gradient_flow}
\dd u - \Div S(\nabla u) \dt &= G(u) \dd W &&\text{ in } \Omega \times \mathcal{O}_T, 
\end{alignat}
with boundary and initial conditions given by 
\begin{alignat}{2}
u &= 0 &&\text{ on } \Omega \times I \times \partial \mathcal{O}, \\
u(0) &= u_0 &&\text{ on } \Omega \times  \mathcal{O}.
\end{alignat}
\end{subequations}
The system \eqref{eq:Gradient_flow} is a perturbed version of the gradient flow of the energy $\mathcal{J}: W^{1,p}_{0,x} \to [0,\infty)$ given by
\begin{align} \label{eq:energy}
\mathcal{J}(u):= \int_{\mathcal{O}} \varphi(\abs{\nabla u}) \dx.
\end{align}
\subsection{Weak and strong solutions} \label{sec:Solution}
We fix the concept of solutions as follows.
\begin{definition} \label{def:sol}
Let $u_0 \in L^2_\omega L^2_x$ be $\mathcal{F}_0$-measurable, $p \in (1,\infty)$ and $(G,W)$ be given by Assumption \ref{ass:Noise}. An $(\mathcal{F}_t)$-adapted process $u \in L^2_x$ is called weak solution to \eqref{eq:p-Laplace} if 
\begin{enumerate}
\item $u \in L^2_\omega C_t L^2_x \cap L^p_\omega L^p_t W^{1,p}_{0,x}$,
\item for all $t \in I$, $\xi \in C^\infty_{0,x}$ and $\mathbb{P}-a.s.$ it holds
\begin{align} \label{eq:p-Laplace_weak}
\int_{\mathcal{O}} (u(t) - u_0) \cdot \xi \dx + \int_0^t \int_{\mathcal{O}} S(\nabla u) :\nabla \xi \dx \ds = \int_{\mathcal{O}} \int_0^t G(u)\dd W_s \cdot \xi \dx.
\end{align}
\end{enumerate}
The process $u$ is called strong solution if it is a weak solution and additionally satisfies
\begin{enumerate}
\item $\Div S(\nabla u) \in L^2_\omega L^2_t L^2_x$ ,
\item for all $t \in I$ and $\mathbb{P}-a.s.$ it holds
\begin{align}
u(t) - u_0 - \int_0^t \Div S(\nabla u) \ds = \int_0^t G(u) \dd W_s
\end{align}
as an equation in $L^2_x$.
\end{enumerate}
\end{definition}

The next step is to answer the question about existence of weak or even strong solutions. Weak solutions can be constructed through a variational approach that uses the monotonicity of the nonlinear diffusion operator $S$ as presented in \cite{LiRo}. In fact, Assumption \ref{ass:Noise} is sufficient to guarantee existence of a unique weak solution. Proving existence of a strong solution is more delicate and usually requires, not only assumptions on the growth of $G$, but also on the gradient of $G$, e.g. for all $x \in \mathcal{O}$, $\xi \in \R^N$
\begin{subequations} \label{ass:grad_growth}
\begin{align}
\sum_{j \in \mathbb{N}} \abs{\nabla_{x} g_j(x,\xi)}^2 &\leq c(1+\abs{\xi})^2, \\
\sum_{j \in \mathbb{N}} \abs{\nabla_{\xi} g_j(x,\xi)}^2 &\leq c.
\end{align}
\end{subequations}
In \cite{Gess2012Strong} a general approach for the construction of strong solutions to gradient flow like equations is presented. In particular, it includes the case of the $p$-Laplace equations.
\begin{theorem}[\cite{Gess2012Strong}, Theorem 4.12] \label{thm:strong-sol}
Assume $p\geq 2$ and $\mathcal{O}$ is a bounded convex domain. Let $(G,W)$ be given by Assumption \ref{ass:Noise}. Additionally, assume \eqref{ass:grad_growth} and $u_0 \in L^{p+\varepsilon}_\omega W^{1,p}_x\cap L_\omega^{2+\varepsilon} L^2_x$ be $\mathcal{F}_0$-measurable for some $\varepsilon > 0$. Then there exists a unique strong solution $u$ to \eqref{eq:p-Laplace}. Moreover,
\begin{align}\label{eq:est-strong}
\mathbb{E}\left[\sup_{t\in I} \norm{u}_{W^{1,p}_x}^p  + \int_0^T \norm{\Div S(\nabla u)}_{L^2_x}^2 \dt \right] \lesssim \mathbb{E} \left[ \norm{u_0}_{W^{1,p}_x}^p \right] + 1.
\end{align}
\end{theorem}

\subsection{Regularity of strong solutions} \label{sec:Regularity}
A key ingredient in the error analysis of numerical algorithms are the improved regularity properties of strong solutions in comparison to those of weak solutions. Concerning time regularity, we prove in \cite{wichmann2021temporal} that the strong solution enjoys $1/2$-time differentiability in an exponential Besov space and even the nonlinear greadient $V(\nabla u)$ obeys $1/2$-time differentiability in a Nikolskii space. The proof uses an assumption on the boundary condition of the noise coefficient $G$, i.e. for all $x \in \partial \mathcal{O}$,
\begin{align}\label{ass:boundary}
\sum_{j\in \mathbb{N}} \abs{g_j(x,0)}^2 = 0.
\end{align}

\begin{theorem}[\cite{wichmann2021temporal} Theorem~3.8 \& Theorem~3.11] \label{thm:time-reg}
Let the assumptions of Theorem \ref{thm:strong-sol} be satisfied. Additionally, assume  \eqref{ass:boundary}. Let $u$ be the strong solution to \eqref{eq:p-Laplace}. Then
\begin{subequations}
\begin{align}
u &\in L^2_\omega B^{1/2}_{\Phi_2,\infty} L^2_x, \\
V(\nabla u) &\in L^2_\omega B^{1/2}_{2,\infty} L^2_x,
\end{align}
\end{subequations}
where $\Phi_2(t) = e^{t^2}-1$.
\end{theorem}

Spatial regularity is closely connected to the existence of strong solutions. Local regularity has been obtained in \cite{Breit2015Regularity}.
\begin{theorem}[\cite{Breit2015Regularity}, Theorem 4]\label{thm:loc-space-reg}
Let $u_0 \in L^2_\omega W^{1,2}_x$ be $\mathcal{F}_0$-measurable. Let Assumption \ref{ass:Noise} be satisfied. Additionally, assume \eqref{ass:grad_growth} and denote by $u$ the strong solution of \eqref{eq:p-Laplace}. Then,
\begin{subequations}
\begin{align}
u &\in L^2_\omega L^\infty_t W^{1,2}_{x,\text{loc}},\\
V(\nabla u) &\in L^2_\omega L^2_t W^{1,2}_{x,\text{loc}}.
\end{align}
\end{subequations}
\end{theorem}
 
On sufficiently regular domains, it is possible to relate the divergence of the nonlinear operator $S$ to the full gradient as presented in \cite{Balci2021,MR4030249}, i.e.
\begin{align}
\norm{\Div S(\nabla u)}_{L^2_\omega L^2_t L^2_x} \eqsim \norm{\nabla S(\nabla u)}_{L^2_\omega L^2_t L^2_x}.
\end{align}
Precisely, given some bounded Lipschitz domain $\mathcal{O}$ such that $\partial \mathcal{O} \in W^{2,1}$, i.e. $\mathcal{O}$ is locally the subgraph of a Lipschitz continuous function of $n-1$ variables, which is also twice weakly differentiable. Denote by $\mathcal{B}$ the second fundamental form on $\partial \mathcal{O}$, by $\abs{B}$ its norm and define 
\begin{align}
\mathcal{K}_\mathcal{O}(r) := \sup_{E \subset \partial \mathcal{O} \cap B_r(x), x \in \partial \mathcal{O}} \frac{\int_E \abs{\mathcal{B}}\dd \mathcal{H}^{n-1}}{\text{cap}_{B_1(x)}(E)},
\end{align}
where $B_r(x)$ denotes the ball of radius $r$ around $x$, $\text{cap}_{B_1(x)}(E)$ is the capacity of the set $E$ relative to the ball $B_1(x)$ and $\mathcal{H}^{n-1}$ is the $n-1$ dimensional Hausdorff measure. 

\begin{lemma}[\cite{MR4030249}, Theorem 2.1] \label{lem:grad-S}
Assume that $\mathcal{O}$ is either 
\begin{enumerate}
\item bounded and convex,
\item or bounded, Lipschitz and $\partial \mathcal{O} \in W^{2,1}$ with $\lim_{r\to 0} \mathcal{K}_\mathcal{O}(r) \leq c$.
\end{enumerate}
Let $v \in L^p_\omega L^p_t W^{1,p}_{0,x}$ with $\Div S(\nabla v) \in L^2_\omega L^2_t L^2_x$. Then $\nabla S(\nabla v) \in L^2_\omega L^2_t L^2_x$ and
\begin{align*}
\mathbb{E} \left[ \norm{\nabla S(\nabla v)}_{L^2_t L^2_x}^2 \right] \eqsim \mathbb{E} \left[ \norm{\Div S(\nabla u)}_{L^2_t L^2_x}^2 \right].
\end{align*}
\end{lemma}

In the non-degenerate setting, $\kappa > 0$, this allows to deduce global spatial regularity.
\begin{corollary}[\cite{wichmann2021temporal}, Corollary~2.14] \label{cor:global-space-reg}
Let the assumptions of Theorem \ref{thm:strong-sol} be satisfied and $\kappa > 0$. Let $u$ be the strong solution to \eqref{eq:p-Laplace}. Then,
\begin{align}
V(\nabla u) &\in L^2_\omega L^2_t W^{1,2}_{x}.
\end{align}
\end{corollary}

\section{Numerical scheme for the averaged system} \label{sec:Numerics}
In this section we will first present the discrete structures. Afterwards we construct two algorithms that approximate the solution to \eqref{eq:p-Laplace}. Finally, we prove convergence of the approximation towards the analytic solution with optimal rates.
\subsection{Space discretization}
Let $\mathcal{O} \subset \R^n$ be a bounded polyhedral domain. By
$\mathcal{T}_h$ denote a regular partition (triangulation)
of~$\mathcal{O}$ (no hanging nodes), which consists of closed $n$-simplices called
\emph{elements}. For each element ($n$-simplex) $T\in \mathcal{T}_h$
we denote by $h_T$ the diameter of $T$, and by $\rho_T$ the supremum
of the diameters of inscribed balls. 

We assume that $\mathcal{T}_h$ is \emph{shape regular}, that is there exists
a constant~$\gamma$ (the shape regularity constant) such that
\begin{align}
  \label{eq:nondeg}
  \max_{T \in \mathcal{T}_h} \frac{h_T}{\rho_T} \leq \gamma.
\end{align}
We define the maximal mesh-size by
\begin{align*}
  h &:= \max_{T \in \mathcal{T}_h} h_T.
\end{align*}
We assume further that our triangulation is \emph{quasi-uniform}, i.e.
\begin{align}
  \label{eq:quasi-uniform}
  h_T \eqsim h \qquad \text{for all $T \in \mathcal{T}_h$}.
\end{align}
%
%

For $\ell\in \setN _0$ we denote by $\mathscr{P}_\ell(\mathcal{O})$ the
polynomials on $\mathcal{O}$ of degree less than or equal to
$\ell$. 

For fixed $r \in \setN$ we define the vector valued finite element space $V_h$
as
\begin{align}\label{def:Vh}
  \begin{aligned}
    V_h &:= \set{v \in W^{1,1}_{0,x} \,:\, v|_T \in \mathscr{P}_r(T)
      \,\,\forall T\in \mathcal{T}_h}.
  \end{aligned}
\end{align}    

 Moreover, let $\Pi_2: L^2_x \to V_h$ be the $L^2_x$-orthogonal projection defined by 
\begin{align*}
\forall \xi_h \in V_h: \quad \left( \Pi_2 v, \xi_h \right) = \left(v,\xi_h \right)
\end{align*}
or equivalently
\begin{align*}
\Pi_2 v = \argmin_{v_h \in V_h} \norm{v-v_h}_{L^2_x}.
\end{align*}
We will need some classical results on the stability properties of the $L^2_x$-projection for finite elements.
\begin{lemma}\label{lem:l2-stab}
Let $r \geq 1$, $V_h$ be given by \eqref{def:Vh} and $\mathcal{T}_h$ be quasi-uniform. Then 
\begin{alignat*}{2}
&\forall v \in W^{1,2}_x: \quad \norm{v-\Pi_2v}_{L^2_x} + h\norm{\nabla(v-\Pi_2 v)}_{L^2_x} &&\lesssim h\norm{\nabla v}_{L^2_x},\\
&\forall v \in W^{2,2}_x: \quad \norm{v-\Pi_2v}_{L^2_x} + h\norm{\nabla(v-\Pi_2 v)}_{L^2_x} &&\lesssim h^2\norm{\nabla^2 v}_{L^2_x}.
\end{alignat*} 
\end{lemma}
Due to the nonlinear structure of the $p$-Laplace system we also need an adapted stability result.
\begin{proposition}[\cite{MR4286257}, Theorem 7]\label{prop:non-stability}
Let $r \geq 1$, $V_h$ be given by \eqref{def:Vh} and $\mathcal{T}_h$ be quasi-uniform. Then
\begin{align*}
\norm{V(\nabla v) - V(\nabla \Pi_2 v)}_{L^2_x} \lesssim h \norm{\nabla V(\nabla v)}_{L^2_x}.
\end{align*}
\end{proposition}

\subsection{Time discretization}
Let $\{0=t_0<\cdots<t_M=T\}$ be a uniform partition of $[0,T]$ with
mesh size $\tau =T/M$. For $m \geq 1$ define $I_m := [t_{m-1},t_m]$. By $\abs{I_m}$ we denote the Lebesgue measure of~$I_m$. By $\dashint_{I_m} g\ds$ we denote the mean value
integral over the set~$I_m$. We also abbreviate $\mean{g}_m = \dashint_{I_m} g\ds$ for the mean value. 

Let us discuss the stability properties of the piecewise constant approximation process generated by the mean values in terms of Nikolskii spaces.
\begin{lemma}\label{lem:approximation quality}
Let $\alpha \in (0,1)$ and $r \in [1,\infty)$. Additionally, assume $u \in L^2_\omega B^{\alpha}_{r,\infty} L^2_x$. Then
\begin{align}\label{eq:approximation quality}
\left( \mathbb{E} \left[ \left( \sum_{m=1}^M \int_{I_m} \norm{u - \mean{u}_m}_{L^2_x}^r \ds \right)^\frac{2}{r} \right] \right)^\frac{1}{2} \lesssim \tau^{\alpha} \left( \mathbb{E}\left[ \seminorm{u}_{B^{\alpha}_{r,\infty} L^2_x}^2 \right] \right)^\frac{1}{2}.
\end{align}
\end{lemma}

\begin{proof}
Due to Jensen's inequality
\begin{align*}
\mathbb{E} \left[ \left( \sum_{m=1}^M \int_{I_m} \norm{u - \mean{u}_m}_{L^2_x}^r \ds \right)^\frac{2}{r} \right] \leq \mathbb{E} \left[ \left(  \sum_{m=1}^M \int_{I_m} \dashint_{I_m} \norm{u(s) - u(t)}_{L^2_x}^r \dt \ds \right)^\frac{2}{r}\right].
\end{align*}
A change of variables $t = s+h$ and Fubini's Theorem yield
\begin{align*}
&\sum_{m=1}^M \tau \dashint_{I_m}  \dashint_{I_m} \norm{u(s) - u(t) }_{L^2_x}^r \dt  \ds \\
&= \sum_{m=1}^M \tau^{-1} \int_{I_m}  \int_{I_m - \set{s}} \norm{u(s) - u(s+h) }_{L^2_x}^r \dd h  \ds \\
&\leq \tau^{-1} \int_0^{2\tau}  \int_{I \cap I- \set{h}} \norm{u(s) - u(s+h) }_{L^2_x}^r \dd s  \dd h \\
&\leq \tau^{-1} \int_0^{2\tau} h^{\alpha r} \dd h \seminorm{u}_{B^{\alpha}_{r,\infty} L^2_x}^r = \frac{2^{\alpha r+1}}{\alpha r +1} \tau^{\alpha r} \seminorm{u}_{B^{\alpha}_{r,\infty} L^2_x}^r.
\end{align*}
Thus, 
\begin{align*}
\left( \mathbb{E} \left[ \left( \sum_{m=1}^M \int_{I_m} \norm{u - \mean{u}_m}_{L^2_x}^r \ds \right)^\frac{2}{r} \right] \right)^\frac{1}{2}&\leq \left( \mathbb{E} \left[ \left(  \frac{2^{\alpha r+1}}{\alpha r +1} \tau^{\alpha r} \seminorm{u}_{B^{\alpha}_{r,\infty} L^2_x}^r \right)^\frac{2}{r}\right] \right)^\frac{1}{2} \\
&\lesssim \tau^\alpha \left( \mathbb{E} \left[\seminorm{u}_{B^{\alpha}_{r,\infty} L^2_x}^2 \right] \right)^\frac{1}{2}.
\end{align*}
The proof is finished.
\end{proof}

\begin{remark}
Lemma \ref{lem:approximation quality} is also valid for $r=\infty$. Here we need to substitute the left hand side in~\eqref{eq:approximation quality} by
\begin{align*}
\left( \mathbb{E} \left[ \max_{m \in \set{1,\ldots,M}} \sup_{s \in I_m} \norm{u(s) - \mean{u}_m}_{L^2_x}^2 \right] \right)^\frac{1}{2}.
\end{align*}
Note, $B^{\alpha}_{\infty,\infty}(I) = C^{\alpha}(I)$ is the space of $\alpha$-H\"older continuous functions.  
\end{remark}

In our application the process $u$ only belongs to $L^2_\omega B^{1/2}_{\Phi_2,\infty} L^2_x \backslash L^2_\omega C^{1/2}_t L^2_x $. Therefore we need to adjust the stability result in terms of exponentially integrable Nikolskii spaces.
\begin{lemma} \label{lem:exponential_stability}
Let $\alpha \in (0,1)$ and $u \in L^2_\omega B^{\alpha}_{\Phi_2,\infty} L^2_x$ where $\Phi_2(t) = e^{t^2}-1$. Then
\begin{align} \label{eq:exponential_stability}
\left( \mathbb{E}\left[ \max_{m\in\set{1,\ldots,M}} \norm{u(t_m) - \mean{u}_m}_{L^2_x}^2 \right] \right)^\frac{1}{2}\lesssim \tau^\alpha \sqrt{\ln(1+\tau^{-1})} \left( \mathbb{E} \left[ \seminorm{u}_{ B^{\alpha}_{\Phi_2,\infty} L^2_x}^2 \right] \right)^\frac{1}{2}. 
\end{align}
\end{lemma}

\begin{proof}
Define $I_m^k:= [t_{m} - 2^{-k}\tau, t_m]$. Due to the embedding $B^{\alpha}_{\Phi_2,\infty} \hookrightarrow C_t$ for all $\alpha > 0$ we find $u$ to be continuous as an $L^2_x$-valued process. Therefore, it holds
\begin{align*}
\lim_{k \to \infty} \mean{u}_{I_m^k} = u(t_m). 
\end{align*}
Observe, $\abs{I_m^{k}} = 2^{-k} \tau$ and $I_m^{k-1} = [t_m - 2^{-(k-1)}\tau,t_m - 2^{-k}\tau]  \cup [t_m - 2^{-k}\tau,t_m]$. A shift in the integral results in 
\begin{align*}
 u(t_m) - \mean{u}_m  &= \sum_{k=1}^\infty \mean{u}_{I_m^k} - \mean{u}_{I_m^{k-1}} \\
&= \sum_{k=1}^\infty 2^{-1} \dashint_{I_m^k} u(s) - u(s - 2^{-k} \tau) \ds.
\end{align*}

Jensen's inequality implies
\begin{align*}
&\mathbb{E} \left[ \max_{m\in\set{1,\ldots,M}} \norm{\mean{u}_m - u(t_m)}_{L^2_x}^2 \right] \\
&\quad \leq \mathbb{E} \left[ \max_{m\in\set{1,\ldots,M}} \left(\sum_{k=1}^\infty 2^{-1} \dashint_{I_m^k} \norm{ u(s) - u(s - 2^{-k} \tau) }_{L^2_x} \ds \right)^2  \right] \\
&\quad \leq \mathbb{E} \left[  \left(\sum_{k=1}^\infty 2^{-1} \max_{m\in\set{1,\ldots,M}} \dashint_{I_m^k} \norm{ u(s) - u(s - 2^{-k} \tau) }_{L^2_x} \ds \right)^2  \right] \\
&\quad \leq \mathbb{E} \left[  \left(\sum_{k=1}^\infty 2^{-1} \lambda_k \max_{m\in\set{1,\ldots,M}} \Phi_2^{-1} \left( \dashint_{I_m^k} \Phi_2 \left( \frac{\norm{ u(s) - u(s - 2^{-k} \tau) }_{L^2_x}}{\lambda_k} \right) \ds \right) \right)^2  \right]
\end{align*}
Choose $\lambda_k $ by 
\begin{align*}
\lambda_k := \inf \left\{ \mu > 0: \int_{I \cap I  - \set{2^{-k}\tau}} \Phi_2 \left( \frac{\norm{u(s) - u(s - 2^{-k}\tau)}_{L^2_x}}{\mu}  \right)\ds \leq 1 \right\}.
\end{align*}
In particular, since $I_m^k \subset I \cap I- \set{2^{-k}\tau}$,
\begin{align*}
\int_{I_m^k} \Phi_2 \left( \frac{\norm{ u(s) - u(s - 2^{-k} \tau ) }_{L^2_x}}{\lambda_k} \right) \ds  \leq 1.
\end{align*}
Thus,
\begin{align*}
&\mathbb{E} \left[  \left(\sum_{k=1}^\infty 2^{-1} \lambda_k \max_{m\in\set{1,\ldots,M}} \Phi_2^{-1} \left( \dashint_{I_m^k} \Phi_2 \left( \frac{\norm{ u(s) - u(s - 2^{-k} \tau) }_{L^2_x}}{\lambda_k} \right) \ds \right) \right)^2  \right] \\
&\quad \leq \mathbb{E} \left[  \left(\sum_{k=1}^\infty 2^{-1} \lambda_k \max_{m\in\set{1,\ldots,M}} \Phi_2^{-1} \left( \abs{I_m^k}^{-1} \right) \right)^2  \right].
\end{align*}
Since $u \in L^2_\omega B^{\alpha}_{\Phi_2,\infty} L^2_x $ it holds
\begin{align*}
\mathbb{E} \left[ \left( \sup_{k\in\mathbb{N}} (2^{-k}\tau)^{-\alpha} \lambda_k \right)^2 \right] \leq \mathbb{E} \left[ \seminorm{u}_{B^{\alpha}_{\Phi_2,\infty} L^2_x}^2 \right].
\end{align*}
It follows
\begin{align}
\begin{aligned}\label{eq:exp_stab1}
&\mathbb{E} \left[  \left(\sum_{k=1}^\infty 2^{-1} \lambda_k \max_{m\in\set{1,\ldots,M}} \Phi_2^{-1} \left( \abs{I_m^k}^{-1} \right) \right)^2  \right] \\
&\quad \leq \mathbb{E} \left[\sup_{k\in\mathbb{N}} \left((2^{-k}\tau)^{-\alpha} \lambda_k \right)^2 \right] \left(\sum_{k=1}^\infty 2^{-1} (2^{-k}\tau)^{\alpha}  \max_{m\in\set{1,\ldots,M}} \Phi_2^{-1} \left( \abs{I_m^k}^{-1} \right) \right)^2 \\
&\quad \leq \mathbb{E} \left[ \seminorm{u}_{B^{\alpha}_{\Phi_2,\infty} L^2_x}^2 \right] \left(\sum_{k=1}^\infty 2^{-1} (2^{-k}\tau)^{\alpha}  \sqrt{\ln (1+ (2^{-k}\tau)^{-1})} \right)^2.
\end{aligned}
\end{align}
Note,
\begin{align*}
\ln \left(1+ (2^{-k} \tau)^{-1} \right) &= \ln(\tau + 2^k) + \ln (\tau^{-1}) \\
&\leq \ln(1 + 2^k) + \ln (1+\tau^{-1}).
\end{align*}
Therefore,
\begin{align}
\begin{aligned}\label{eq:exp_stab2}
&\sum_{k=1}^\infty 2^{-1} (2^{-k}\tau)^{\alpha}  \sqrt{\ln (1+ (2^{-k}\tau)^{-1})}  \\
&\quad \leq \tau^{\alpha} \sum_{k=1}^\infty 2^{-1} 2^{-\alpha k} \left(  \sqrt{\ln (1+ 2^k)} + \sqrt{\ln(1+\tau^{-1})}  \right) \\
&\quad\leq \tau^{\alpha}\left(1 + \sqrt{\ln(1+\tau^{-1})} \right) \sum_{k=1}^\infty 2^{-1} 2^{-\alpha k}  \sqrt{\ln (1+ 2^k)} \\
&\quad\lesssim \tau^{\alpha} \sqrt{\ln(1+\tau^{-1})} .
\end{aligned}
\end{align}
The assertion follows by applying the square root in~\eqref{eq:exp_stab1} and using the estimate~\eqref{eq:exp_stab2}. 
\end{proof}

\subsection{The averaged algorithm}
Let $u$ be the strong solution of \eqref{eq:p-Laplace}, i.e. for all $t \in I$ and $\mathbb{P}$-a.s.
\begin{align} \label{eq:strong}
u(t) - u_0 - \int_0^t \Div S(\nabla u_\nu) \dd \nu - \int_0^t G(u_\nu) \dd W_\nu =0.
\end{align}
Take the average over $I_1$ in \eqref{eq:strong} to get
\begin{align} \label{eq:aver-strong1}
\mean{u}_1 - u_0 - \dashint_{I_1} \left\{ \int_0^t \Div S(\nabla u_\nu) \dd \nu + \int_0^t G(u_\nu) \dd W_\nu  \right\} \dt = 0.
\end{align}
To obtain the general evolution of the time averaged values of the solution we subtract \eqref{eq:strong} for $t$ and $t-\tau$ and take the average over $I_m$
\begin{align} \label{eq:aver-strongn}
\mean{u}_m - \mean{u}_{m-1} - \dashint_{I_m}  \left\{ \int_{t-\tau}^t \Div S(\nabla u_\nu) \dd \nu + \int_{t-\tau}^t G(u_\nu) \dd W_\nu  \right\} \dt = 0.
\end{align}
We denote by $(\cdot,\cdot)$ the $L^2_x$-inner product. The corresponding weak formulation reads for all $\xi \in L^2_x \cap W^{1,p}_{0,x}$
\begin{align} \label{eq:weak}
\left( \mean{u}_m - \mean{u}_{m-1}, \xi \right) + \dashint_{I_m}  \int_{t-\tau}^t  \left(  S(\nabla u_\nu)  , \nabla \xi \right) \dd \nu \dt = \left( \dashint_{I_m}  \int_{t-\tau}^t G(u_\nu) \dd W_\nu \dt, \xi \right).
\end{align}
Due to the (stochastic) Fubini's Theorem we can equivalently write
\begin{subequations}
\begin{align} \label{eq:aver-strongweigth0}
\left( \mean{u}_1 - u_0, \xi \right) + \int_\R  a_{0}(\nu)  \left(  S(\nabla u_\nu)  , \nabla \xi \right) \dd \nu  &= \left(\int_\R a_{0}(\nu) G(u_\nu) \dd W_\nu , \xi \right),\\ \label{eq:aver-strongweigthm}
\left( \mean{u}_m - \mean{u}_{m-1}, \xi \right) + \int_\R  a_{m-1}(\nu)  \left(  S(\nabla u_\nu)  , \nabla \xi \right) \dd \nu  &= \left(\int_\R a_{m-1}(\nu) G(u_\nu) \dd W_\nu , \xi \right),
\end{align}
\end{subequations}
where the weights are given by
\begin{subequations}
\begin{align}\label{eq:weigths-0}
a_0(\nu) &:= \frac{\nu}{\tau} 1_{I_1}(\nu),\\
a_{m-1}(\nu) &:= \frac{\nu - t_{m-2}}{\tau} 1_{I_{m-1}}(\nu) + \frac{t_m - \nu}{\tau} 1_{I_{m}}(\nu).
\end{align}
\end{subequations}
The above considerations motivate the construction of the following numerical scheme:\\
\begin{subequations} \label{eq:algo}
We initialize the algorithm as 
\begin{align}\label{eq:num0}
v_0 := \Pi_2 u_0.
\end{align} 
In order to accurately reflect the special character of the first step \eqref{eq:aver-strong1}, we define $v_1$ via
\begin{align} \label{eq:num1}
\left(v_1 -  v_0, \xi_h \right) + \frac{\tau}{2} \left( S(\nabla v_1), \nabla \xi_h \right) &= \left(G(v_0)\mean{W}_1 , \xi_h \right).
\end{align}
Moreover the evolution for $m \in \set{2,\ldots,M}$ is determinded via
\begin{align} \label{eq:num2}
\left(v_m - v_{m-1}, \xi_h \right) + \tau \left( S(\nabla v_m), \nabla \xi_h \right) = \left(G(v_{m-2}) \left(\mean{W}_m - \mean{W}_{m-1} \right), \xi_h \right)
\end{align}
for all $\xi_h \in V_h$ and $\mathbb{P}$-a.s.
\end{subequations}

The need of a special step size in the first step~\eqref{eq:num1} might be undesirable. It can be overcome by performing a full initial step. We propose the following second algorithm.

\begin{subequations} \label{algo:2nd}
Initialize
\begin{align}\label{algo:2nd-0}
w_0 := \Pi_2 u_0.
\end{align}
Full initial step
\begin{align} \label{algo:2nd-1}
\left(w_1 -  w_0, \xi_h \right) + \tau \left( S(\nabla w_1), \nabla \xi_h \right) &= \left(G(w_0)\mean{W}_1 , \xi_h \right).
\end{align}
Time stepping, for $m \geq 2$,
\begin{align} \label{algo:2nd-2}
\left(w_m -  w_{m-1}, \xi_h \right) + \tau \left( S(\nabla w_m), \nabla \xi_h \right) &= \left(G(w_{m-2})(\mean{W}_{m}-\mean{W}_{m-1}) , \xi_h \right).
\end{align}
\end{subequations}

 The additional difficulties in the error analysis only occur in the estimate of the initial error. 

The next theorem ensures that the numerical schemes are well defined. The arguments are rather standard and we only refer to \cite{MR3607728} Section 3 for a detailed analysis of a more general system.
\begin{theorem}\label{thm:Existence}
Let $u_0 \in L^2_\omega L^2_x$ be $\mathcal{F}_0$-measurable and Assumption \ref{ass:Noise} be satisfied. Then for all $M \in \mathbb{N}$ there exists a unique $\bfv = (v_0,\ldots,v_M) \in (V_h)^{M+1}$ such that $v_m$ is $\mathcal{F}_{t_m}$-measurable for all $m \in \set{0,\ldots,M}$ and $\bfv$ solves \eqref{eq:algo} respectively \eqref{algo:2nd}.
\end{theorem}

\subsection{Error analysis}
We are ready to state and prove the main result.
\begin{theorem}[Convergence of Algorithm \eqref{eq:algo}] \label{thm:Convergence}
Let $u_0 \in L^2_\omega L^2_x$ be $\mathcal{F}_0$-measurable and Assumption \ref{ass:Noise} be satisfied. Additionally, assume $\mathcal{T}_h$ to be quasi-uniform. Moreover, let $u$ be the weak solution to \eqref{eq:p-Laplace} and $\bfv$ be the numerical solution of \eqref{eq:algo}. 
\begin{enumerate}
\item \label{en:Convergence1}
Assume 
\begin{subequations}
\begin{align} \label{ass:reg-sol}
 u &\in L^2_\omega B^{1/2}_{2,\infty} L^2_x \cap L^2_\omega L^\infty_t W^{1,2}_x, \\
 V(\nabla u) &\in L^2_\omega B^{1/2}_{2,\infty} L^2_x \cap L^2_\omega L^2_t W^{1,2}_x.
\end{align}
\end{subequations}
Then it holds
\begin{align}
\begin{aligned} \label{eq:Rates}
&\mathbb{E} \left[ \max_{m=1,\ldots,M} \norm{\mean{u}_m - v_m}_{L^2_x}^2 + \sum_{m=1}^M\int_{I_m} \norm{V(\nabla u_\nu) - V(\nabla v_m)}_{L^2_x}^2 \dd \nu   \right] \\
&\hspace{3em}\lesssim \tau  + h^2.
\end{aligned}
\end{align}
\item \label{en:Convergence2}
If additionally 
\begin{align} \label{ass:time-reg}
u \in L^2_\omega B^{1/2}_{\Phi_2,\infty} L^2_x,
\end{align}
then
\begin{align}
\begin{aligned} \label{eq:Rates2}
\mathbb{E} \left[ \max_{m=1,\ldots,M} \norm{u(t_m)- v_m}_{L^2_x}^2 \right] \lesssim \tau \ln(1+\tau^{-1})  + h^2.
\end{aligned}
\end{align}
\end{enumerate}
\end{theorem}

\begin{remark}
The regularity assumption~\eqref{ass:time-reg} is optimal in the sense, that it is the limiting space of the time regularity of the Wiener process $W$. Hyt\"onen and Veraar prove in~\cite{VerHyt08} that a Banach space valued Brownian motion has full $1/2$-differentiability only on the Nikolskii-scale. More precisely, they show for all $q,r \in [1,\infty)$ and $\mathbb{P}$-a.s. 
\begin{align*}
W \in B^{1/2}_{\Phi_2, \infty} \quad \text{ and } \quad  W \not \in B^{1/2}_{q,r}.
\end{align*}

The logarithmic term in~\eqref{eq:Rates2} has already been used in the context of rough stochastic differential equations~\cite{MR4125787,LL21}. It quantifies the distance of $L^{\Phi_2}$ and $L^\infty$. 

Theorem~\ref{thm:Convergence} can be generalized to cover fractional time and space regularity assumptions as done for the deterministic $p$-Laplace system in~\cite{MR4286257}.
\end{remark}

\begin{proof}
\textbf{Part \ref{en:Convergence1}:} Fix $m \in \set{2,\ldots,M}$, subtract \eqref{eq:weak} from \eqref{eq:num2} and choose $\xi_h = \Pi_2 e_m := \Pi_2 \mean{u}_m - v_m $
\begin{align*}
H_1 + H_2 &:= \left(e_m - e_{m-1}, \Pi_2 e_m   \right) + \dashint_{I_m}  \int_{t-\tau}^t  \left(  S(\nabla u_\nu)  - S(\nabla v_m) , \nabla \Pi_2 e_m \right) \dd \nu \dt\\
& = \left( \dashint_{I_m}  \int_{t-\tau}^t [G(u_\nu)  - G(v_{m-2})] \dd W_\nu \dt, \Pi_2 e_m  \right) =: H_3.
\end{align*}
Now use the symmetry of the $L^2_x$-projection and the algebraic identity $2a(a-b) = a^2 - b^2 + (a-b)^2$ to get
\begin{align*}
H_1 &= \left(\Pi_2 (e_m - e_{m-1}), \Pi_2 e_m   \right) \\
&= \frac{1}{2} \left(\norm{\Pi_2 e_m }_{L^2_x}^2 -\norm{\Pi_2 e_{m-1} }_{L^2_x}^2 + \norm{\Pi_2[ e_m- e_{m-1}] }_{L^2_x}^2    \right).
\end{align*}
The second term, due to the $V$-coercivity \eqref{eq:V-coercive},
\begin{align*}
H_2 &=  \dashint_{I_m}  \int_{t-\tau}^t  \left(  S(\nabla u_\nu)  - S(\nabla v_m) , \nabla \Pi_2 \mean{u}_m - \nabla v_m \right) \dd \nu \dt \\
&\eqsim \dashint_{I_m}  \int_{t-\tau}^t  \norm{V(\nabla u_\nu) - V(\nabla v_m)}_{L^2_x}^2 \dd \nu \dt \\
&\quad + \dashint_{I_m}  \int_{t-\tau}^t  \left(  S(\nabla u_\nu)  - S(\nabla v_m) , \nabla \Pi_2 \mean{u}_m - \nabla u_\nu \right) \dd \nu \dt.
\end{align*}
Summation over $m \in \set{2,\ldots,m^*}$ with $m^* \leq M$ results in
\begin{align*}
\norm{\Pi_2 e_{m^*}}_{L^2_x}^2 &+ \sum_{m=2}^{m^*} \norm{\Pi_2[ e_m- e_{m-1}] }_2^2 + \sum_{m=2}^{m^*}\dashint_{I_m}  \int_{t-\tau}^t  \norm{V(\nabla u_\nu) - V(\nabla v_m)}_{L^2_x}^2 \dd \nu \dt \\
&\eqsim \norm{\Pi_2 e_1}_{L^2_x}^2 + \sum_{m=2}^{m^*} \left( \dashint_{I_m}  \int_{t-\tau}^t [G(u_\nu)  - G(v_{m-2})] \dd W_\nu \dt, \Pi_2 e_m  \right) \\
&\quad - \sum_{m=2}^{m^*} \dashint_{I_m}  \int_{t-\tau}^t  \left(  S(\nabla u_\nu)  - S(\nabla v_m) , \nabla \Pi_2 \mean{u}_m - \nabla u_\nu \right) \dd \nu \dt.
\end{align*}

Take the maximum over $m^* \in \set{2,\ldots,M}$ and expectation
\begin{align}
\begin{aligned} \label{eq:big_est}
&\mathbb{E} \left[ \max_{m^* \in \set{2,\ldots,M}} \norm{\Pi_2 e_{m^*}}_{L^2_x}^2 \right] + \mathbb{E} \left[\sum_{m=2}^M \norm{\Pi_2[ e_m- e_{m-1}] }_{L^2_x}^2 \right] \\
&\quad + \mathbb{E} \left[\sum_{m=2}^M \dashint_{I_m}  \int_{t-\tau}^t  \norm{V(\nabla u_\nu) - V(\nabla v_m)}_{L^2_x}^2 \dd \nu \dt \right] \\
&\lesssim \mathbb{E} \left[ \norm{\Pi_2 e_1}_{L^2_x}^2 \right] + \mathbb{E} \left[ \max_{{m^*} \in \set{2,\ldots,M}} \sum_{m=2}^{m^*} \left( \dashint_{I_m}  \int_{t-\tau}^t [G(u_\nu)  - G(v_{m-2})] \dd W_\nu \dt, \Pi_2 e_m  \right)   \right]\\
& \quad + \mathbb{E} \left[ \sum_{m=2}^M \dashint_{I_m}  \int_{t-\tau}^t  \abs{ \left(  S(\nabla u_\nu)  - S(\nabla v_m) , \nabla \Pi_2 \mean{u}_m - \nabla u_\nu \right) } \dd \nu \dt\right] \\
&=: J_1 + J_2 + J_3.
\end{aligned}
\end{align}

\textbf{Step 1:} We start by estimating the initial error. Similarly as before, we subtract the weak formulation of \eqref{eq:aver-strong1} from \eqref{eq:num1} and choose $\xi_h = \Pi_2 e_1$
\begin{align*}
\left( e_1 - e_0, \Pi_2 e_1 \right) &+ \dashint_{I_1} \int_0^t \left( S(\nabla u_\nu) - S(\nabla v_1), \nabla \Pi_2 \mean{u}_1 - \nabla v_1 \right) \dd \nu \dt \\
&\quad = \left( \dashint_{I_1} \int_0^t G(u_\nu) - G(v_0) \dd W_\nu \dt, \Pi_2 e_1 \right),
\end{align*}
where $e_0 := u_0 - v_0$. By \eqref{eq:num0} we have $\Pi_2 e_0 = 0$. Therefore,
\begin{align}
\begin{aligned} \label{eq:init_est}
&\mathbb{E}\left[ \norm{\Pi_2 e_1}_{L^2_x}^2 \right] + \mathbb{E} \left[ \dashint_{I_1} \int_0^t \norm{V(\nabla u_\nu) - V(\nabla v_1)}_{L^2_x}^2 \dd \nu \dt \right] \\
&\quad \lesssim \mathbb{E} \left[ \left( \dashint_{I_1} \int_0^t G(u_\nu) - G(v_0) \dd W_\nu \dt, \Pi_2 e_1 \right) \right] \\
&\quad \quad + \mathbb{E}\left[ \dashint_{I_1} \int_0^t \left( S(\nabla u_\nu) - S(\nabla v_1), \nabla u_\nu - \nabla \Pi_2 \mean{u}_1 \right) \dd \nu \dt \right]\\ 
&\quad =: J_{1,a} + J_{1,b}.
\end{aligned}
\end{align}
Due to H\"older's and Young's inequalities and It\^o isometry
\begin{align*}
J_{1,a} &\leq \frac{1}{2} \mathbb{E} \left[\norm{   \dashint_{I_1} \int_0^t G(u_\nu) - G(v_0) \dd W_\nu  \dt }_{L^2_x}^2\right] + \frac{1}{2} \mathbb{E} \left[ \norm{\Pi_2 e_1}_{L^2_x}^2 \right] \\
&= \frac{1}{2} \mathbb{E} \left[  \int_0^\tau a_0(\nu)^2 \norm{ G(u_\nu) - G(v_0)}_{L_2(U;L^2_x)}^2 \dd \nu   \right] + \frac{1}{2} \mathbb{E} \left[ \norm{\Pi_2 e_1}_{L^2_x}^2 \right].
\end{align*}
Absorb the second term to the left hand side. For the first term we use $a_0 \leq 1$ and the Lipschitz assumption \eqref{ass:Lipschitz}. Now, since the operator norm of the $L^2_x$-projection is bounded by one, 
\begin{align} \nonumber
&\mathbb{E} \left[  \int_0^\tau a_0(\nu)^2 \norm{ G(u_\nu) - G(v_0)}_{L_2(U;L^2_x)}^2 \dd \nu   \right] \\ \label{eq:init_error}
&\quad \lesssim \mathbb{E} \left[  \int_0^\tau \norm{u_\nu-v_0}_{L^2_x}^2 \dd \nu  \right] \lesssim \tau \left( \mathbb{E} \left[ \norm{u}_{L^\infty_t L^2_x}^2  \right] + \mathbb{E} \left[ \norm{ u_0}_{L^2_x}^2 \right] \right).
\end{align}
Thanks to the generalized Young's inequality, cf. Lemma \ref{lem:gen-young},
\begin{align*}
J_{1,b} &\leq \delta \mathbb{E} \left[ \dashint_{I_1} \int_0^t \norm{V(\nabla u_\nu) - V(\nabla v_1)}_{L^2_x}^2 \dd \nu \dt \right] \\
&\quad + c_\delta \mathbb{E} \left[\dashint_{I_1}\dashint_{I_1} \int_0^t \norm{V(\nabla u_\nu) - V(\nabla \Pi_2 u_s)}_{L^2_x}^2 \dd \nu \dt \ds \right].
\end{align*}
Absorb the first term to the left hand side in \eqref{eq:init_est}. The second is split up into a space and a time error. Using the nonlinear stability of the $L^2_x$-projection, cf. Proposition~\ref{prop:non-stability},
\begin{align*}
&\mathbb{E} \left[\dashint_{I_1}\dashint_{I_1} \int_0^t \norm{V(\nabla u_\nu) - V(\nabla \Pi_2 u_s)}_{L^2_x}^2 \dd \nu \dt \ds \right]\\
&\lesssim \mathbb{E} \left[\dashint_{I_1}\dashint_{I_1} \int_0^t \norm{V(\nabla u_\nu) - V(\nabla u_s)}_{L^2_x}^2 \dd \nu \dt \ds \right] \\
&\quad + \mathbb{E} \left[\dashint_{I_1}\dashint_{I_1} \int_0^t \norm{V(\nabla u_s) - V(\nabla \Pi_2 u_s)}_{L^2_x}^2 \dd \nu \dt \ds \right] \\
&\lesssim \tau \mathbb{E} \left[ \seminorm{V(\nabla u)}_{B^{1/2}_{2,\infty} L^2_x}^2 \right]  + h^2 \mathbb{E} \left[ \norm{\nabla V(\nabla u)}_{L^2_t L^2_x}^2 \right].
\end{align*}
Overall, we arrive at the estimate from~\eqref{eq:init_est}
\begin{align}\label{eq:J1}
\begin{aligned}
&J_1 +(1-\delta) \mathbb{E} \left[ \dashint_{I_1} \int_0^t \norm{V(\nabla u_\nu) - V(\nabla v_1)}_{L^2_x}^2 \dd \nu \dt \right] \\
&\lesssim \tau \left( \mathbb{E} \left[ \norm{u}_{L^\infty_t L^2_x}^2  \right] + \mathbb{E} \left[ \norm{ u_0}_{L^2_x}^2 \right] \right) \\
&\quad + c_\delta\left( \tau \mathbb{E} \left[ \seminorm{V(\nabla u)}_{B^{1/2}_{2,\infty} L^2_x}^2 \right]  + h^2 \mathbb{E} \left[ \norm{\nabla V(\nabla u)}_{L^2_t L^2_x}^2 \right] \right).
\end{aligned}
\end{align}

\textbf{Step 2:} The stochastic part in~\eqref{eq:big_est} needs some refined analysis
\begin{align*}
J_2 &\leq \mathbb{E} \left[ \max_{m^* \in \set{2,\ldots,M}} \sum_{m=2}^{m^*} \left( \dashint_{I_m}  \int_{t-\tau}^t [G(u_\nu)  - G(v_{m-2})] \dd W_\nu \ds, \Pi_2 (e_m- e_{m-2})  \right)   \right] \\
&\quad + \mathbb{E} \left[ \max_{m^* \in \set{2,\ldots,M}} \sum_{m=2}^{m^*} \left( \dashint_{I_m}  \int_{t-\tau}^t [G(u_\nu)  - G(v_{m-2})] \dd W_\nu \ds, \Pi_2 e_{m-2}  \right)   \right]\\
&=: J_{2,a} + J_{2,b}.
\end{align*}
The first, due to Hölder's and Young's inequalities and an index shift,
\begin{align*}
J_{2,a} &\leq \mathbb{E} \left[ \sum_{m=2}^M \left( c_\varepsilon \norm{ \dashint_{I_m}  \int_{t-\tau}^t [G(u_\nu)  - G(v_{m-2})] \dd W_\nu \ds }_{L^2_x}^2 + \varepsilon\norm{\Pi_2 (e_m- e_{m-2})}_{L^2_x}^2 \right) \right]\\
&\lesssim c_\varepsilon \mathbb{E} \left[ \sum_{m=2}^M \norm{ \dashint_{I_m}  \int_{t-\tau}^t [G(u_\nu)  - G(v_{m-2})] \dd W_\nu \ds }_{L^2_x}^2  \right] \\
&\quad + \varepsilon \left( \mathbb{E} \left[ \sum_{m=2}^M  \norm{\Pi_2 (e_m- e_{m-1})}_{L^2_x}^2  \right] + \mathbb{E} \left[ \norm{\Pi_2 e_1}_{L^2_x}^2 \right] \right).
\end{align*}
The second term can be absorbed to the left hand side in \eqref{eq:big_est}. The third term is nothing but the initial error $J_1$. For the first term we invoke It\^o isometry, the Lipschitz condition \eqref{ass:Lipschitz} and $a_{m-1} \leq 1$
\begin{align*}
&\mathbb{E} \left[ \sum_{m=2}^M \norm{ \dashint_{I_m}  \int_{t-\tau}^t [G(u_\nu)  - G(v_{m-2})] \dd W_\nu \ds }_{L^2_x}^2  \right] \\
&\quad = \sum_{m=2}^M  \int_{t_{m-2}}^{t_m} a_{m-1}^2(\nu) \mathbb{E} \left[   \norm{ G(u_\nu)  - G(v_{m-2})}_{L_2(\mathcal{U};L^2_x)}^2\right] \dd \nu \\
&\quad \lesssim  \sum_{m=2}^M  \int_{t_{m-2}}^{t_m} \mathbb{E} \left[   \norm{u_\nu - v_{m-2}}_{L^2_x}^2 \right] \dd \nu.
\end{align*}
Decomposition in time and space error, applying Lemma~\ref{lem:l2-stab} and Lemma~\ref{lem:approximation quality} and the estimate~\eqref{eq:init_error}
\begin{align}
\begin{aligned} \label{eq:est1}
&\sum_{m=2}^M  \int_{t_{m-2}}^{t_m} \mathbb{E} \left[   \norm{u_\nu - v_{m-2}}_{L^2_x}^2 \right] \dd \nu\\
&\quad \lesssim   \sum_{m=3}^M  \int_{t_{m-2}}^{t_m}  \mathbb{E} \left[   \norm{u_\nu - \mean{u}_{m-2}}_{L^2_x}^2 \right] \dd \nu + \tau \sum_{m=3}^M    \mathbb{E} \left[   \norm{\mean{u}_{m-2}- \Pi_2 \mean{u}_{m-2}}_{L^2_x}^2 \right]   \\
&\quad \quad + \tau\sum_{m=3}^M  \mathbb{E} \left[   \norm{  \Pi_2 \mean{u}_{m-2} - v_{m-2}}_{L^2_x}^2 \right] + \int_0^{2\tau} \mathbb{E} \left[ \norm{u_\nu - v_0}_{L^2_x}^2 \right] \dd \nu  \\
&\quad \lesssim \tau  \mathbb{E} \left[ \seminorm{u}_{ B^{1/2}_{2,\infty} L^2_x}^2 \right] +  h^2 \mathbb{E} \left[   \norm{  \nabla u }_{L^2_t L^2_x}^2 \right]  + \tau\sum_{m=3}^M  \mathbb{E} \left[   \norm{  \Pi_2 e_{m-2} }_{L^2_x}^2 \right]\\
&\quad\quad  + \tau \left( \mathbb{E} \left[ \norm{u}_{L^\infty_t L^2_x}^2  \right] + \mathbb{E} \left[ \norm{ u_0}_{L^2_x}^2 \right] \right).
\end{aligned}
\end{align}

Next, we analyze the second term~$J_{2,b}$ in the upper bound for $J_2$. Define the discrete real-valued stochastic process
\begin{align} \label{def:Mart_not}
K_{m^*} &:=  \sum_{m=2}^{m^*} \left( \dashint_{I_m}  \int_{t-\tau}^t [G(u_\nu)  - G(v_{m-2})] \dd W_\nu \dt, \Pi_2 e_{m-2}  \right).
\end{align}
It is convenient to use stochastic Fubini's theorem to rewrite
\begin{align*}
K_{m^*} = \sum_{m=2}^{m^*} \left( \int_{t_{m-2}}^{t_m} a_{m-1}(\nu) [G(u_\nu)  - G(v_{m-2})] \dd W_\nu, \Pi_2 e_{m-2}  \right).
\end{align*}

In the following we abbreviate $\mathcal{F}_m:= \mathcal{F}_{t_m}$. Note, \eqref{def:Mart_not} does not define a martingale with respect to $\mathcal{F}_{t_m}$. In fact, the discrepancy of not being a martingale can be quantified. The general strategy is to split up the sum into a martingale and an error term. The error term is called compensator. To determine how the compensator looks, we first compute the conditional expectations of $K_M$ with respect to $\mathcal{F}_{m^*}$, i.e.,
\begin{align*}
&\mathbb{E} \left[K_M \big| \mathcal{F}_{m^*} \right] = \mathbb{E} \left[\sum_{m = m^* + 2}^M  \left( \int_{t_{m-2}}^{t_m} a_{m-1}(\nu) [G(u_\nu)  - G(v_{m-2})] \dd W_\nu, \Pi_2 e_{m-2}  \right) \Big| \mathcal{F}_{m^*} \right] \\
&\quad \quad + \mathbb{E} \left[ \left( \int_{t_{m^*-1}}^{t_{m^* + 1}} a_{m^*}(\nu) [G(u_\nu)  - G(v_{m^*-1})] \dd W_\nu, \Pi_2 e_{m^*-1}  \right) \right] + \mathbb{E} \left[ K_{m^*}\big| \mathcal{F}_{m^*} \right] \\
&\quad =:\mathcal{M}_1 + \mathcal{M}_2 +\mathcal{M}_3.
\end{align*}
Due to the tower property of conditional expectation, the measurability of $e_m$ with respect to $\mathcal{F}_m$ together with the martingale property of the stochastic integral
\begin{align*}
\mathcal{M}_1 &= \sum_{m = m^* + 2}^M \mathbb{E} \left[  \mathbb{E} \left[\left(  \int_{t_{m-2}}^{t_m} a_{m-1}(\nu) [G(u_\nu)  - G(v_{m-2})] \dd W_\nu , \Pi_2 e_{m-2}  \right) \Big| \mathcal{F}_{m-2} \right] \Big| \mathcal{F}_{m^*} \right] \\
&= \sum_{m = m^* + 2}^M \mathbb{E} \left[  \left(  \mathbb{E} \left[ \int_{t_{m-2}}^{t_m} a_{m-1}(\nu) [G(u_\nu)  - G(v_{m-2})] \dd W_\nu  \Big| \mathcal{F}_{m-2} \right], \Pi_2 e_{m-2}  \right) \Big| \mathcal{F}_{m^*} \right] \\
&=0.
\end{align*}
Again using the $\mathcal{F}_m$-measurability of $e_m$, we conclude that $K_{m^*}$ is $\mathcal{F}_{m^*}$-measurable. Thus,
\begin{align*}
\mathcal{M}_3 &= K_{m^*}.
\end{align*}
It remains to compute the conditional expectation in~$\mathcal{M}_2$. Since $t_{m^*}$ is an interior point of $I_{m^*+1} \cup I_{m^*} $ we split up the stochastic integral into a part that only sees values above the threshold $t_{m^*}$ and into a lower part that only sees values below the threshold $t_{m^*}$, i.e.,
\begin{align*}
\mathcal{M}_2 &= \mathbb{E} \left[ \left( \int_{t_{m^*}}^{t_{m^* + 1}} a_{m^*}(\nu)  [G(u_\nu)  - G(v_{m^*-1})] \dd W_\nu , \Pi_2 e_{m^*-1}  \right)\big| \mathcal{F}_{m^*} \right]\\
&+\mathbb{E} \left[ \left( \int_{t_{m^*-1}}^{t_{m^*}} a_{m^*}(\nu) [G(u_\nu)  - G(v_{m^*-1})] \dd W_\nu , \Pi_2 e_{m^*-1}  \right)\big| \mathcal{F}_{m^*} \right].
\end{align*}
The first vanishes due to the martingale property of the stochastic integral, while the second is measurable with respect to $\mathcal{F}_{m^*}$. Overall,
\begin{align*}
\mathcal{M}_2 &= \left( \int_{t_{m^*-1}}^{t_{m^*}} a_{m^*}(\nu) [G(u_\nu)  - G(v_{m^*-1})] \dd W_\nu , \Pi_2 e_{m^*-1}  \right).
\end{align*}
$\mathcal{M}_2$ is called compensator and quantifies the error of not being a martingale, i.e.,
\begin{align} \label{eq:Mart_not}
\mathbb{E} \left[K_M \big| \mathcal{F}_{m^*} \right]-K_{m^*} = \mathcal{M}_2.
\end{align}

Furthermore, increments of the discrete stochastic process $K$ satisfy 
\begin{align}\label{eq:inc-K}
K_{m^*} - K_{{m^*}-1} &= \left( \int_{t_{m^*-2}}^{t_{m^*}} a_{m^*-1}(\nu) [G(u_\nu)  - G(v_{{m^*}-2})] \dd W_\nu , \Pi_2 e_{{m^*}-2}  \right).
\end{align}
\eqref{eq:Mart_not} together with \eqref{eq:inc-K} allow to identify increments of the conditional expectations, 
\begin{align*}
&\mathbb{E} \left[K_M \big| \mathcal{F}_{m^*} \right] - \mathbb{E} \left[K_M \big| \mathcal{F}_{{m^*}-1} \right] \\
&\quad =\left( \int_{t_{m^*-1}}^{t_{m^*}} a_{m^*}(\nu) [G(u_\nu)  - G(v_{m^*-1})] \dd W_\nu , \Pi_2 e_{m^*-1}  \right) \\
&\quad \quad + \left( \int_{t_{m^*-1}}^{t_{m^*}} a_{m^*-1}(\nu) [G(u_\nu)  - G(v_{{m^*}-2})] \dd W_\nu , \Pi_2 e_{{m^*}-2}  \right).
\end{align*}
Observe $\mathbb{E}\left[ K_M \big| \mathcal{F}_1 \right] = 0$, since $\Pi_2 e_0 = 0$. The Burkholder-Davis-Gundy's inequality implies
\begin{align*}
&\mathbb{E} \left[\max_{{m^*} \in \set{2,\ldots,M}} \abs{ \mathbb{E}\left[K_M \big| \mathcal{F}_{m^*} \right]  }\right] \lesssim \mathbb{E} \left[ \left( \sum_{m=2}^M \left( \mathbb{E}\left[K_M \big| \mathcal{F}_m \right] - \mathbb{E}\left[K_M \big| \mathcal{F}_{m-1} \right] \right)^2 \right)^\frac{1}{2}  \right]  \\
&\quad \lesssim \mathbb{E} \left[ \left( \sum_{m=2}^M \norm{ \int_{t_{m-1}}^{t_m} a_m(\nu) [G(u_\nu)  - G(v_{m-1})] \dd W_\nu }_{L^2_x}^2 \norm{ \Pi_2 e_{m-1} }_{L^2_x}^2\right)^\frac{1}{2} \right] \\
&\quad +\mathbb{E} \left[ \left( \sum_{m=2}^M \norm{ \int_{t_{m-1}}^{t_m} a_{m-1}(\nu)[G(u_\nu)  - G(v_{m-2})] \dd W_\nu }_{L^2_x}^2 \norm{ \Pi_2 e_{m-2}  }_{L^2_x}^2 \right)^\frac{1}{2} \right].
\end{align*}

Now Young's inequality, It\^o isometry and the Lipschitz condition~\eqref{ass:Lipschitz} imply
\begin{align*}
&\mathbb{E} \left[\max_{{m^*} \in \set{2,\ldots,M}} \mathbb{E}\left[K_M \big| \mathcal{F}_{m^*} \right]  \right] \\
&\quad \lesssim \varepsilon \mathbb{E}\left[\max_{{m^*} \in \set{2,\ldots,M}} \norm{\Pi_2 e_{m^*}}_{L^2_x}^2  +  \norm{\Pi_2 e_1}_{L^2_x}^2  \right] \\
&\quad + c_\varepsilon \mathbb{E} \left[ \sum_{m=2}^M \int_{t_{m-1}}^{t_m} \norm{G(u_\nu)  - G(v_{m-2})}_{L_2(U;L^2_x)}^2\dd \nu \right] \\
&\quad +c_\varepsilon\mathbb{E} \left[ \sum_{m=1}^M \int_{t_{m-1}}^{t_m}  \norm{G(u_\nu)  - G(v_{m-1})}_{L_2(U;L^2_x)}^2 \dd \nu  \right]\\
&\lesssim \varepsilon \mathbb{E}\left[\max_{{m^*} \in \set{2,\ldots,M}} \norm{\Pi_2 e_{m^*}}_{L^2_x}^2   +  \norm{\Pi_2 e_1}_{L^2_x}^2  \right] \\
&\quad + c_\varepsilon \mathbb{E} \left[ \sum_{m=2}^M \int_{t_{m-1}}^{t_m} \norm{u_\nu  - v_{m-2}}_{L^2_x}^2 \dd \nu + \sum_{m=1}^M \int_{t_{m-1}}^{t_m} \norm{u_\nu  - v_{m-1}}_{L^2_x}^2 \dd \nu  \right].
\end{align*}
The second term is estimated as in \eqref{eq:est1}
\begin{align*}
 &\mathbb{E} \left[ \sum_{m=2}^M \int_{t_{m-1}}^{t_m} \norm{u_\nu  - v_{m-2}}_{L^2_x}^2 \dd \nu + \sum_{m=1}^M \int_{t_{m-1}}^{t_m} \norm{u_\nu  - v_{m-1}}_{L^2_x}^2 \dd \nu  \right]\\
&\lesssim \tau  \mathbb{E} \left[ \seminorm{u}_{ B^{1/2}_{2,\infty} L^2_x}^2 \right] +  h^2 \mathbb{E} \left[   \norm{  \nabla u }_{L^2_t L^2_x}^2 \right]  + \tau\sum_{m=3}^M  \mathbb{E} \left[   \norm{  \Pi_2 e_{m-2} }_{L^2_x}^2 \right]\\
&\quad + \tau \left( \mathbb{E} \left[ \norm{u}_{L^\infty_t L^2_x}^2  \right] + \mathbb{E} \left[ \norm{ u_0}_{L^2_x}^2 \right] \right).
\end{align*}
Similarly, due to \eqref{eq:Mart_not}, H\"older's inequality, Young's inequality and $\ell^1 \hookrightarrow \ell^\infty$
\begin{align*}
&\mathbb{E} \left[\max_{ {m^*} \in \set{2,\ldots,M}} \left( K_{m^*} - \mathbb{E}\left[K_M \big| \mathcal{F}_{m^*} \right]  \right) \right] \\
&= \mathbb{E} \left[\max_{{m^*} \in \set{2,\ldots,M}}  \left( \int_{t_{m^*-1}}^{t_{m^*}} a_{m^*}(\nu) [G(u_\nu)  - G(v_{m^*-1})] \dd W_\nu , \Pi_2 e_{m^*-1}  \right)\right] \\
&\leq \mathbb{E} \left[\max_{{m^*} \in \set{2,\ldots,M}} \norm{ \int_{t_{m^*-1}}^{t_{m^*}} a_{m^*}(\nu) [G(u_\nu)  - G(v_{m^*-1})] \dd W_\nu}_{L^2_x} \norm{ \Pi_2 e_{{m^*}-1} }_{L^2_x}\right] \\
&\leq \varepsilon  \mathbb{E} \left[\max_{{m^*} \in \set{2,\ldots,M}}  \norm{ \Pi_2 e_{{m^*}-1} }_{L^2_x}^2\right] \\
&\quad + c_\varepsilon  \mathbb{E} \left[\max_{{m^*} \in \set{2,\ldots,M}} \norm{ \int_{t_{m^*-1}}^{t_{m^*}} a_{m^*}(\nu) [G(u_\nu)  - G(v_{m^*-1})] \dd W_\nu}_{L^2_x}^2\right]  \\
&\leq \varepsilon  \mathbb{E} \left[\max_{{m^*} \in \set{2,\ldots,M}}  \norm{ \Pi_2 e_{{m^*}} }_{L^2_x}^2  +\norm{ \Pi_2 e_{1} }_{L^2_x}^2  \right] \\
&\quad + c_\varepsilon  \mathbb{E} \left[\sum_{m=2}^M \norm{   \int_{t_{m-1}}^{t_{m}} a_{m}(\nu) [G(u_\nu)  - G(v_{m-1})] \dd W_\nu}_{L^2_x}^2\right].
\end{align*}
Together we can estimate
\begin{align*}
J_{2,b} &= \mathbb{E} \left[\max_{{m^*} \in \set{2,\ldots,M}} K_{m^*} \right] \\
&= \mathbb{E} \left[\max_{{m^*} \in \set{2,\ldots,M}} \left( K_{m^*} - \mathbb{E}\left[K_M \big| \mathcal{F}_{m^*} \right]  + \mathbb{E}\left[K_M \big| \mathcal{F}_{m^*} \right] \right) \right] \\
&\leq \mathbb{E} \left[\max_{{m^*} \in \set{2,\ldots,M}} \left( K_{m^*} - \mathbb{E}\left[K_M \big| \mathcal{F}_{m^*} \right]  \right) \right] + \mathbb{E} \left[\max_{{m^*} \in \set{2,\ldots,M}} \mathbb{E}\left[K_M \big| \mathcal{F}_{m^*} \right]  \right] \\
&\lesssim \varepsilon \mathbb{E}\left[\max_{{m^*} \in \set{2,\ldots,M}} \norm{\Pi_2 e_{m^*}}_{L^2_x}^2 +  \norm{\Pi_2 e_1}_{L^2_x}^2 \right] +c_\varepsilon \mathbb{E} \left[\tau \sum_{m=2}^M \norm{\Pi_2 e_{m-1}}_{L^2_x}^2 \right] \\
&\quad + c_\varepsilon \left( \tau  \mathbb{E} \left[ \seminorm{u}_{ B^{1/2}_{2,\infty} L^2_x}^2 \right] +  h^2 \mathbb{E} \left[   \norm{  \nabla u }_{L^2_t L^2_x}^2 \right]  + \tau  \mathbb{E} \left[ \norm{u}_{L^\infty_t L^2_x}^2  \right] + \tau \mathbb{E} \left[ \norm{ u_0}_{L^2_x}^2 \right] \right).
\end{align*}
This concludes the bound for $J_2$ in~\eqref{eq:big_est}
\begin{align} \label{eq:estimate-J2}
\begin{aligned}
J_2 &\lesssim \varepsilon \mathbb{E}\left[\max_{ {m^*} \in \set{2,\ldots,M}} \norm{\Pi_2 e_{m^*}}_{L^2_x}^2 \right] + \varepsilon\mathbb{E} \left[ \sum_{m=2}^M  \norm{\Pi_2 (e_m- e_{m-1})}_{L^2_x}^2  \right] + \varepsilon J_1 \\
&\quad + c_\varepsilon \left( \tau  \mathbb{E} \left[ \seminorm{u}_{ B^{1/2}_{2,\infty} L^2_x}^2 \right] +  h^2 \mathbb{E} \left[   \norm{  \nabla u }_{L^2_t L^2_x}^2 \right]  + \tau  \mathbb{E} \left[ \norm{u}_{L^\infty_t L^2_x}^2  \right] + \tau \mathbb{E} \left[ \norm{ u_0}_{L^2_x}^2 \right] \right) \\
&\quad + c_\varepsilon \mathbb{E} \left[ \tau \sum_{m=2}^M  \norm{\Pi_2 e_m}_{L^2_x}^2 \right].
\end{aligned}
\end{align}

\textbf{Step 3:} In this step we estimate the nonlinear gradient. Jensen's inequality, the generalized Young's inequality~\eqref{eq:gen-young} and the nonlinear stability result Proposition~\ref{prop:non-stability} imply
\begin{align*}
J_3 &= \mathbb{E} \left[ \sum_{m=2}^M \dashint_{I_m}  \int_{t-\tau}^t  \abs{ \left(  S(\nabla u_\nu)  - S(\nabla v_m) , \nabla \Pi_2 \mean{u}_m - \nabla u_\nu \right) } \dd \nu \dt\right] \\
&\leq \mathbb{E}\left[ \sum_{m=2}^M  \dashint_{I_m} \dashint_{I_m}  \int_{t-\tau}^t  \abs{ \left(  S(\nabla u_\nu)  - S(\nabla v_m) , \nabla \Pi_2 u_{\nu_2} - \nabla u_\nu \right) }  \dd \nu_2 \dd \nu \dt \right] \\
&\leq \varepsilon \mathbb{E} \left[ \sum_{m=2}^M \dashint_{I_m}  \int_{t-\tau}^t  \norm{V(\nabla u_\nu) - V(\nabla v_m)}_{L^2_x}^2 \dd \nu \dt\right] \\
&\quad+ c_\varepsilon \mathbb{E} \left[ \sum_{m=2}^M \dashint_{I_m}  \int_{t-\tau}^t  \norm{V(\nabla u_\nu) - V(\nabla \Pi_2 u_\nu)}_{L^2_x}^2 \dd \nu \dt \right]\\
&\quad+ c_\varepsilon \mathbb{E} \left[ \sum_{m=2}^M  \dashint_{I_m}\dashint_{I_m}  \int_{t-\tau}^t  \norm{ V(\nabla  u_\nu) -  V(\nabla u_{\nu_2})}_{L^2_x}^2 \dd \nu_2 \dd \nu \dt \right]\\
&\leq \varepsilon \mathbb{E} \left[ \sum_{m=2}^M \dashint_{I_m}  \int_{t-\tau}^t  \norm{V(\nabla u_\nu) - V(\nabla v_m)}_{L^2_x}^2 \dd \nu \dt\right] \\
&\quad + h^2 c_\varepsilon \mathbb{E} \left[ \norm{\nabla V(\nabla u)}_{L^2_t L^2_x}^2 \right] + \tau c_\varepsilon \mathbb{E}\left[ \seminorm{V(\nabla u)}_{ B^{1/2}_{2,\infty} L^2_x}^2\right].
\end{align*}

\textbf{Step 4:} We aim at applying a Gronwall type argument. Collecting all estimates we get
\begin{align*}
&(1- \varepsilon) \mathbb{E} \left[ \max_{ {m^*} \in \set{1,\ldots,M}} \norm{\Pi_2 e_{m^*}}_{L^2_x}^2 \right] + (1-\varepsilon)\mathbb{E} \left[\sum_{m=1}^M \norm{\Pi_2[ e_m- e_{m-1}] }_{L^2_x}^2 \right] \\
&\quad + (1-\varepsilon) \mathbb{E} \left[\sum_{m=1}^M \dashint_{I_m}  \int_{{t-\tau} \vee 0}^t  \norm{V(\nabla u_\nu) - V(\nabla v_m)}_2^2 \dd \nu \dt \right] \\
&\lesssim  c_\varepsilon \mathbb{E} \left[ \sum_{m=1}^M \tau  \norm{\Pi_2 e_m}_{L^2_x}^2 \right] + c_\varepsilon \left( \tau \mathbb{E} \left[ \seminorm{V(\nabla u)}_{B^{1/2}_{2,\infty} L^2_x}^2 \right]  + h^2 \mathbb{E} \left[ \norm{\nabla V(\nabla u)}_{L^2_t L^2_x}^2 \right] \right) \\
&\quad +c_\varepsilon \left( \tau  \mathbb{E} \left[ \seminorm{u}_{ B^{1/2}_{2,\infty} L^2_x}^2 \right] +  h^2 \mathbb{E} \left[   \norm{  \nabla u }_{L^2_t L^2_x}^2 \right]  + \tau  \mathbb{E} \left[ \norm{u}_{L^\infty_t L^2_x}^2  \right] + \tau \mathbb{E} \left[ \norm{ u_0}_{L^2_x}^2 \right] \right).
\end{align*}
Choosing $\varepsilon$ sufficiently small and applying Gronwall's Lemma ensures
\begin{align*}
\mathbb{E} \left[ \max_{ {m^*} \in \set{1,\ldots,M}} \norm{\Pi_2 e_{m^*}}_{L^2_x}^2 \right] \lesssim  e^{cT} \left( \tau  + h^2 \right).
\end{align*}
This implies
\begin{align}\label{eq:step4}
\mathbb{E} \left[ \max_{ {m^*} \in \set{1,\ldots,M}} \norm{\Pi_2 e_{m^*}}_{L^2_x}^2 +\sum_{m=1}^M \dashint_{I_m}  \int_{{t-\tau} \vee 0}^t  \norm{V(\nabla u_\nu) - V(\nabla v_m)}_{L^2_x}^2 \dd \nu \dt \right] \lesssim \tau + h^2.
\end{align}

\textbf{Step 5:} We artificially introduce the desired error quantities and use the stability of mean-value time projections and the $L^2$-space projection. Let us denote $\mean{f}_{a_m} := \dashint_{I_m} \dashint_{t-\tau \vee 0}^{t} f_\nu \dd \nu \dt $. 
A slight modification of Lemma~\ref{lem:approximation quality} implies
\begin{align*}
\mathbb{E} \left[\sum_{m=1}^M \int_{I_m} \norm{V(\nabla u_\nu) - \mean{V(\nabla u)}_{a_m} }_{L^2_x}^2 \dd \nu  \right] &\lesssim \tau \mathbb{E} \left[ \seminorm{V(\nabla u)}_{B^{1/2}_{2,\infty} L^2_x}^2 \right].
\end{align*}
Additionally, Jensen's inequality ensures
\begin{align*}
&\mathbb{E} \left[\sum_{m=1}^M \tau \norm{\mean{V(\nabla u)}_{a_m} - V(\nabla v_m)}_{L^2_x}^2 \right] \\
&\quad \leq \mathbb{E} \left[\sum_{m=1}^M \dashint_{I_m}  \int_{{t-\tau} \vee 0}^t  \norm{V(\nabla u_\nu) - V(\nabla v_m)}_{L^2_x}^2 \dd \nu \dt \right].
\end{align*}
Therefore, invoking \eqref{eq:step4},
\begin{align*}
 &\mathbb{E} \left[\sum_{m=1}^M \int_{I_m} \norm{V(\nabla u_\nu) - V(\nabla v_m)}_{L^2_x}^2 \dd \nu  \right] \\
 &\hspace{3em}\lesssim \mathbb{E} \left[\sum_{m=1}^M \int_{I_m} \norm{V(\nabla u_\nu) - \mean{V(\nabla u)}_{a_m} }_{L^2_x}^2 \dd \nu  \right]  \\
 &\hspace{3em}\quad + \mathbb{E} \left[\sum_{m=1}^M \tau \norm{\mean{V(\nabla u)}_{a_m} - V(\nabla v_m)}_{L^2_x}^2 \right] \\
 &\hspace{3em}\lesssim \tau \mathbb{E} \left[ \seminorm{V(\nabla u)}_{B^{1/2}_{2,\infty} L^2_x}^2 \right] + \tau + h^2.
\end{align*}
The assertion~\eqref{eq:Rates} follows by an application of Lemma~\ref{lem:l2-stab}
\begin{align*}
&\mathbb{E} \left[ \max_{m \in \set{1,\ldots,M}} \norm{\mean{u}_m - v_m}_{L^2_x}^2 \right] \\
&\quad \lesssim \mathbb{E} \left[ \max_{m \in \set{1,\ldots,M}} \norm{\Pi_2 e_m}_{L^2_x}^2 \right] + \mathbb{E} \left[ \max_{m \in \set{1,\ldots,M}} \norm{\mean{u}_m - \Pi_2 \mean{u}_m }_{L^2_x}^2 \right] \\
&\quad \lesssim \tau + h^2 + h^2 \mathbb{E} \left[ \norm{\nabla u}_{L^\infty_t L^2_x}^2 \right].
\end{align*}

\textbf{Part \ref{en:Convergence2}:} Now, let us assume $u \in L^2_\omega B^{1/2}_{\Phi_2,\infty} L^2_x$. We apply Lemma~\ref{lem:exponential_stability} to bound
\begin{align*}
&\mathbb{E}\left[\max_{m \in \set{1,\ldots,M}} \norm{u(t_m) - v_m}_{L^2_x}^2 \right] \\
&\quad \lesssim \mathbb{E}\left[\max_{m \in \set{1,\ldots,M}} \norm{\mean{u}_m - v_m}_{L^2_x}^2 \right] + \mathbb{E}\left[\max_{m \in \set{1,\ldots,M}} \norm{u(t_m) - \mean{u}_m}_{L^2_x}^2 \right] \\
&\quad \lesssim \tau + h^2 + \tau \ln(1+\tau^{-1}).
\end{align*}
This verifies \eqref{eq:Rates2} and the proof is finished.
\end{proof}


One can also use time averages on the nonlinear gradient term to measure the error of the approximation.
\begin{corollary}\label{cor:aver_conv}
Let the assumptions of Theorem~\ref{thm:Convergence}~\ref{en:Convergence1} be satisfied. Then
\begin{align} \label{eq:aver_conv1}
\mathbb{E}\left[ \sum_{m=1}^M \tau \norm{\mean{V(\nabla u)}_m - V(\nabla v_m)}_{L^2_x}^2 \right] \lesssim \tau + h^2
\end{align}
and
\begin{align}\label{eq:aver_conv2}
\mathbb{E}\left[ \sum_{m=1}^M \tau \norm{V(\nabla \mean{u}_m) - V(\nabla v_m)}_{L^2_x}^2 \right] \lesssim \tau + h^2.
\end{align}
\end{corollary}
\begin{proof}
The estimate~\eqref{eq:aver_conv1} immediately follows by an application of Jensen's inequality and the bound \eqref{eq:Rates},
\begin{align*}
\mathbb{E}\left[ \sum_{m=1}^M \tau \norm{\mean{V(\nabla u)}_m - V(\nabla v_m)}_{L^2_x}^2 \right] &\leq \mathbb{E}\left[ \sum_{m=1}^M \int_{I_m} \norm{V(\nabla u_\nu) - V(\nabla v_m)}_{L^2_x}^2 \dd \nu \right] \\
&\lesssim \tau + h^2.
\end{align*}
In order to prove the second estimate~\eqref{eq:aver_conv2} we use Lemma~\ref{lem:V-coercive} and Lemma~\ref{lem:gen-young}
\begin{align*}
&\mathbb{E}\left[ \sum_{m=1}^M \tau \norm{V(\nabla \mean{u}_m) - V(\nabla v_m)}_{L^2_x}^2 \right] \\
&\hspace{3em} \eqsim \mathbb{E}\left[ \sum_{m=1}^M \int_{I_m}  \int_\mathcal{O} (S(\nabla \mean{u}_m) - S(\nabla v_m) ): (\nabla u_\nu - \nabla v_m) \dx \dd \nu \right] \\
&\hspace{3em} \leq c \mathbb{E}\left[ \sum_{m=1}^M \int_{I_m} \norm{V(\nabla u_\nu) - V(\nabla v_m)}_{L^2_x}^2 \dd \nu \right] \\
&\hspace{3em} \quad + \frac{1}{2} \mathbb{E}\left[ \sum_{m=1}^M \tau \norm{V(\nabla \mean{u}_m) - V(\nabla v_m)}_{L^2_x}^2 \right].
\end{align*}
Absorbing the second term to the left hand side and applying the bound~\eqref{eq:Rates} verifies the assertion.
\end{proof}

\begin{remark}
Although both \eqref{eq:aver_conv1} and \eqref{eq:aver_conv2} enjoy the same convergence rates, it is not clear whether one dominates the other. In the linear case, $p=2$, the terms coincide. 
\end{remark}

The averaged error quantities~\eqref{eq:aver_conv1} and \eqref{eq:aver_conv2} are equivalent up to oscillation to the error quantity~\eqref{eq:Rates}.
\begin{lemma} \label{lem:equiv-error-quant}
Let $u \in L^p_t W^{1,p}_x$ and $A \in L^p_x$. Then
\begin{align} \label{eq:equiv-error-quant1}
\begin{aligned}
&\dashint_{I_m} \norm{V( \nabla u_\nu) - V(A)}_{L^2_x}^2  \dd \nu \\
&\hspace{3em}= \dashint_{I_m} \norm{V( \nabla u_\nu ) - \mean{V(\nabla u)}_{m} }_{L^2_x}^2 \dd \nu +  \norm{\mean{V(\nabla u)}_{m}  - V(A)}_{L^2_x}^2
\end{aligned}
\end{align}
and
\begin{align}\label{eq:equiv-error-quant2}
\begin{aligned}
&\dashint_{I_m} \norm{V( \nabla u_\nu) - V(A)}_{L^2_x}^2  \dd \nu \\
&\hspace{3em}\lesssim \dashint_{I_m} \norm{V( \nabla u_\nu ) - V(\nabla \mean{u}_{m})  }_{L^2_x}^2 \dd \nu +  \norm{V(\nabla \mean{u }_{m}) - V(A)}_{L^2_x}^2.
\end{aligned}
\end{align}
\end{lemma}
\begin{proof}
The equation~\eqref{eq:equiv-error-quant1} follows by using 
\begin{align*}
\dashint_{I_m} \left(V(\nabla u_\nu) - \mean{V(\nabla u)}_m,\mean{V(\nabla u)}_m - V(A)  \right) \dd \nu = 0.
\end{align*}
The estimate~\eqref{eq:equiv-error-quant2} is obtained trivially,
\begin{align*}
&\dashint_{I_m} \norm{V( \nabla u_\nu) - V(A)}_{L^2_x}^2  \dd \nu \\
&\hspace{3em} = \dashint_{I_m} \norm{V( \nabla u_\nu) - V(\nabla \mean{u}_m)- (V(\nabla \mean{u}_m) -  V(A))}_{L^2_x}^2  \dd \nu\\
&\hspace{3em}\lesssim \dashint_{I_m} \norm{V( \nabla u_\nu ) - V(\nabla \mean{u}_{m})  }_{L^2_x}^2 \dd \nu +  \norm{V(\nabla \mean{u }_{m}) - V(A)}_{L^2_x}^2.
\end{align*}
\end{proof}

\begin{remark}
In \cite[Lemma~6.2]{MR2847446}  the authors prove the equivalence (although it is done purely in space but can be extended to time) of
\begin{align*}
\dashint_{I_m} \norm{V( \nabla u_\nu ) - V(\nabla \mean{u}_{m})  }_{L^2_x}^2 \dd \nu \eqsim \dashint_{I_m} \norm{V( \nabla u_\nu ) - \mean{V(\nabla u)}_{m} }_{L^2_x}^2 \dd \nu .
\end{align*}
If $V(\nabla u) \in B^{1/2}_{2,\infty} L^2_x $, then Lemma~\ref{lem:approximation quality} implies
\begin{align*}
\sum_{m=1}^M \dashint_{I_m} \norm{V( \nabla u_\nu ) - \mean{V(\nabla u)}_{m} }_{L^2_x}^2 \dd \nu \lesssim \tau \seminorm{V(\nabla u)}_{B^{1/2}_{2,\infty} L^2_x}^2.
\end{align*}
\end{remark}

\begin{theorem}[Convergence of Algorithm \ref{algo:2nd}] \label{thm:Convergence2}
Let the assumptions of Theorem~\ref{thm:Convergence} be satisfied. Denote by $\bfw \in (V_h)^{M+1}$ the solution to \eqref{algo:2nd} and by $u$ the weak solution to \eqref{eq:p-Laplace}.
Then 
\begin{align}
\begin{aligned} \label{eq:Rates-algo-2nd}
&\mathbb{E} \left[ \max_{m=1,\ldots,M} \norm{\mean{u}_m - w_m}_{L^2_x}^2 + \sum_{m=1}^M\int_{I_m} \norm{V(\nabla u_\nu) - V(\nabla w_m)}_{L^2_x}^2 \dd \nu   \right] \\
&\hspace{3em}\lesssim \tau  + h^2
\end{aligned}
\end{align}
and
\begin{align}\label{eq:Rates-algo-2nd-2}
\mathbb{E} \left[ \max_{m=1,\ldots,M} \norm{u(t_m) - w_m}_{L^2_x}^2 \right] \lesssim \tau \ln(1+\tau^{-1}) + h^2.
\end{align}
\end{theorem}

\begin{proof}
The proof proceeds similarly to the proof of Theorem \ref{thm:Convergence}. We will only prove the bound for the initial error. Instead of comparing $w_1$ to $\mean{u}_1$, we rather choose $\mean{u}_{I_1\cup I_2} := \dashint_{I_1 \cup I_2} u_\nu \dd \nu$. The equation for the latter one reads
\begin{align} \label{eq:fullstep}
\left( \mean{u}_{I_1 \cup I_2} - u_0, \xi \right) + \dashint_{I_1 \cup I_2} \int_0^t \left( S(\nabla u_\nu), \nabla \xi \right) \dd \nu \dt =\left( \dashint_{I_1 \cup I_2} \int_0^t G(u_\nu) \dd W_\nu \dt, \xi \right).
\end{align}
Subtracting \eqref{algo:2nd-1} from \eqref{eq:fullstep} and choosing $\xi_h = \Pi_2 \mean{u}_{I_1 \cup I_2} - w_1$ results in
\begin{align*}
A_1 + A_2 &:=\norm{\Pi_2 \mean{u}_{I_1 \cup I_2} - w_1}_{L^2_x}^2 \\
&\quad + \dashint_{I_1 \cup I_2} \int_0^t \left( S(\nabla u_\nu) - S(\nabla w_1), \nabla (\Pi_2 \mean{u}_{I_1 \cup I_2} - w_1) \right) \dd \nu \dt \\
&= \left( \dashint_{I_1 \cup I_2} \int_0^t G(u_\nu) \dd W_\nu \dt,\Pi_2 \mean{u}_{I_1 \cup I_2} - w_1 \right)  \\
&\quad - \left( G(w_0) \mean{W}_1, \Pi_2 \mean{u}_{I_1 \cup I_2} - w_1  \right)\\
&=: A_3 + A_4.
\end{align*}
Due to H\"older's and Young's inequalities, It\^o isometry and the growth assumption \eqref{ass:growth}
\begin{align*}
\mathbb{E} \left[ A_3 \right] &\leq \frac{1}{4} \mathbb{E} \left[ A_1 \right] + \mathbb{E} \left[ \norm{\dashint_{I_1 \cup I_2} \int_0^t G(u_\nu) \dd W_\nu \dt}_{L^2_x}^2 \right]  \\
&= \frac{1}{4} \mathbb{E} \left[ A_1 \right] + \mathbb{E} \left[\int_{I_1 \cup I_2} \left( \frac{\nu}{2\tau} \right)^2  \norm{G(u_\nu)}_{L_2(U;L^2_x)}^2 \dd \nu \right] \\
&\leq \frac{1}{4} \mathbb{E} \left[ A_1 \right] + c \mathbb{E} \left[\int_{I_1 \cup I_2} \left( \frac{\nu}{2\tau} \right)^2  \norm{1+u_\nu}_{L^2_x}^2 \dd \nu \right].
\end{align*}
The boundedness of $u$ as an $L^2_x$ valued-process implies
\begin{align*}
\mathbb{E} \left[ A_3 \right] \leq \frac{1}{4} \mathbb{E} \left[ A_1 \right] + c \frac{2}{3}\tau \norm{1+u}_{L^2_\omega L^\infty_t L^2_x}^2.
\end{align*}
The forth term is estimated similarly,
\begin{align*}
\mathbb{E}\left[ A_4\right] &\leq \frac{1}{4} \mathbb{E} \left[ A_1 \right] + \mathbb{E} \left[ \norm{G(w_0) \mean{W}_1}_{L^2_x}^2 \right] \\
&\leq \frac{1}{4} \mathbb{E} \left[ A_1 \right] + c \frac{1}{3}\tau  \mathbb{E} \left[ \norm{1+\Pi_2 u_0}_{L^2_x}^2 \right].
\end{align*}
Using the $L^2$-stability of the $L^2$-projection we obtain
\begin{align*}
\mathbb{E} \left[ A_4 \right] \leq \frac{1}{4} \mathbb{E} \left[ A_1 \right] + c \frac{1}{3}\tau \norm{1+u_0}_{L^2_\omega L^2_x}^2.
\end{align*}

It remains to check the second term. Here we use the same arguments as in step~3 of the proof of Theorem~\ref{thm:Convergence} to conclude
\begin{align*}
A_2 &\geq \int_{I_1} \norm{V(\nabla u_\nu) - V(\nabla w_1)}_{L^2_x}^2 \dd \nu - c\left(  \tau \seminorm{V(\nabla u)}_{B^{1/2}_{2,\infty} L^2_x}^2 + h^2 \norm{\nabla V(\nabla u)}_{L^2_t L^2_x}^2 \right).
\end{align*}
Overall, we have established
\begin{align*}
&\mathbb{E} \left[ \norm{\Pi_2 \mean{u}_{I_1 \cup I_2} - w_1}_{L^2_x}^2\right] + \mathbb{E} \left[ \int_{I_1} \norm{V(\nabla u_\nu) - V(\nabla w_1)}_{L^2_x}^2 \dd \nu \right] \\
&\lesssim \tau \mathbb{E} \left[ \norm{1+u}_{L^\infty_t L^2_x}^2 + \norm{1+u_0}_{L^2_x}^2 + \seminorm{V(\nabla u)}_{B^{1/2}_{2,\infty}L^2_x}^2 \right] + h^2 \mathbb{E} \left[ \norm{\nabla V(\nabla u)}_{L^2_t L^2_x}^2 \right].
\end{align*}
Lastly, using Lemma~\ref{lem:l2-stab} and Lemma~\ref{lem:approximation quality}
\begin{align*}
&\mathbb{E} \left[ \norm{\mean{u}_1 - w_1}_{L^2_x}^2 \right] \\
&\quad \lesssim \mathbb{E} \left[ \norm{\mean{u}_1 - \mean{u}_{I_1 \cup I_2}}_{L^2_x}^2 \right] +\mathbb{E} \left[ \norm{ \mean{u}_{I_1 \cup I_2} - \Pi_2 \mean{u}_{I_1 \cup I_2}}_{L^2_x}^2 \right]\\
&\quad \quad +\mathbb{E} \left[ \norm{\Pi_2 \mean{u}_{I_1 \cup I_2} - w_1}_{L^2_x}^2 \right] \\
&\quad \lesssim \tau \mathbb{E} \left[\norm{1+u}_{L^\infty_t L^2_x}^2 + \norm{1+u_0}_{L^2_x}^2 + \seminorm{u}_{B^{1/2}_{2,\infty} L^2_x}^2 + \seminorm{V(\nabla u)}_{B^{1/2}_{2,\infty}L^2_x}^2    \right] \\
&\quad \quad + h^2 \mathbb{E}\left[ \norm{\nabla u}_{L^\infty_t L^2_x}^2 +\norm{\nabla V(\nabla u)}_{L^2_t L^2_x}^2\right] .
\end{align*}
The bound for the initial error is complete.

\end{proof}

\section{Discrete stochastic processes} \label{sec:Discrete}
In this section we investigate the law of averaged Wiener processes and propose an implementable sampling algorithm.

\subsection{Wiener process}
We introduce the concept of Hilbert space valued Gaussian processes.
\begin{definition}
A stochastic process $Y$ is called $U$-valued Gaussian process with mean operator $m: I \times U \to \R $ and variance operator $\Sigma: I \times U \times U \to \R$, if for all $t \in I$ and $u \in U$ it holds 
\begin{align*}
\varphi_{Y_t}(u):= \mathbb{E}\left[ e^{-i \left(Y_t, u \right)_U} \right] = e^{-i m_t(u) - \frac{1}{2} \Sigma_t(u,u)}.
\end{align*}
In short, we write $Y_t \sim \mathcal{N}\left(m_t, \Sigma_t \right)\footnote[1]{denotes equality in distribution}$.
\end{definition}
It follows that the stochastic forcing $W$ defined by~\eqref{rep:W} is an $U$-valued Gaussian process with mean operator $m_t(u) = 0$ and variance operator $\Sigma_t(u,v) := t (u,v)_U$ for all $t\in I$ and $u,v \in U$. Moreover, the series \eqref{rep:W} converges in the weak topology of $U$, due to the hypercontractivity of normally distributed random variables it holds for any $q >0$, $u \in U$ and $t,s \in I$
\begin{align*}
\left( \mathbb{E} \left[ \abs{\left( W_t - W_s, u \right)_U}^q \right] \right)^\frac{1}{q} &= \left( \mathbb{E} \left[ \abs{ \sum_{j\in \mathbb{N}} \left(u_j , u \right)_U \left( \beta_t^j - \beta_s^j\right) }^q \right] \right)^\frac{1}{q}\\
&\eqsim \sqrt{q} \abs{t-s}^\frac{1}{2} \norm{u}_U.
\end{align*}
In fact, as soon as the index set is infinite we lose the norm convergence, since
\begin{align*}
\left( \mathbb{E}\left[ \norm{W_t}_U^q \right] \right)^\frac{1}{q} \eqsim \sqrt{q}  \left(t \sum_{j\in \mathbb{N}} \norm{u_j}_U^2 \right)^\frac{1}{2} = \infty.
\end{align*}

%
%
\subsection{Averaged Wiener process}
In this section we compute the distributions of the random variables $(\mean{W}_m)_{m=1}^M$ and the joint distribution of the averaged increments. 

A key tool in the derivation of the distribution of the averaged Wiener process is the decomposition of the process $W$ adjusted to the equidistant time partition $\set{I_m}_{m=1}^M$. We decompose $W|_{I_m}$ into a Brownian bridge $\mathcal{B}_m$ and its nodal values $W(t_{m-1})$ and $W(t_{m})$, i.e., for $t \in I_m$
\begin{align}\label{W-decomp}
W(t) = W(t_{m-1}) + \mathcal{B}_m(t) + \frac{t - t_{m-1}}{\tau}\Delta_m W, 
\end{align}
where 
\begin{align}\label{def:Brownian-bridge}
\mathcal{B}_m(t) := W(t) - W(t_{m-1}) - \frac{t - t_{m-1}}{\tau}\Delta_m W.
\end{align}
Brownian bridges have nice independency properties. They do not look into the past nor future. The following result can be found in~\cite[Section~1.2]{MR2454984}.
\begin{proposition}
Let $(\mathcal{B}_m)_{m=1}^M$ be given by~\eqref{def:Brownian-bridge}. Then for all $m \in \set{1,\ldots,M}$
\begin{align}
\sigma\left( \mathcal{B}_m(t) \big| t \in I_m \right) \perp\sigma\left( W(t) \big| t\in [0,\infty) \backslash (t_{m-1},t_m) \right),
\end{align}
i.e. all finite dimensional distributions of the generators of each sigma algebra are independent of each other. 
\end{proposition}

\begin{corollary}\label{cor:Independency-BB}
Let $(\mathcal{B}_m)_{m=1}^M$ be given by~\eqref{def:Brownian-bridge}. Then $\mathcal{B}_1, \ldots, \mathcal{B}_M, \Delta_1 W, \ldots, \Delta_M W$ are independent.
\end{corollary}

Next, we take the time average over the interval $I_m$ in~\eqref{W-decomp} and obtain
\begin{align} \label{def:Mean-decomp}
\mean{W}_m =  W(t_{m-1}) + \mean{\mathcal{B}_m}_m + \frac{\Delta_m W}{2}.
\end{align}

Now, it is our choice whether we want to compute the distribution of $\mean{W}_m$ or $\mean{\mathcal{B}_m}_m$. The formula~\eqref{def:Mean-decomp} provides an easy way to compute the remaining one. We choose to compute the distribution of $\mean{W}_m$. An application of It\^o's formula for $f(s,W_s) = \frac{s-t_m}{\tau} W_s$ implies $\mathbb{P}$-a.s.
\begin{align}
\mean{W}_m = W_{t_{m-1}}+  \int_{t_{m-1}}^{t_m} \frac{t_m - s}{\tau} \dd W_s.
\end{align}
A stochastic integral that is driven by a Wiener process and a deterministic integrand stays Gaussian. 
\begin{lemma}[\cite{MR3726894} Prop. 7.1]\label{lem:inv-Gauss}
Let $f \in L^2(I)$. Then 
\begin{align*}
t \mapsto \int_0^t f_s \dd W_s 
\end{align*}
is an $U$-valued Gaussian process with zero mean and variance
\begin{align*}
\Sigma_t (u,v) = \int_0^t \abs{f_s}^2 \ds \left(u,v \right)_U. 
\end{align*}
\end{lemma}
\begin{corollary}\label{cor:Dist-meanW}
$\mean{W}_m$ is an $U$-valued Gaussian random variable with zero mean and variance $\Sigma(u,v) = \frac{2t_{m-1} + t_m}{3}(u,v)_U $.
\end{corollary}
\begin{proof}
Let us define
\begin{align*}
\mean{W}_m = W_{t_{m-1}}+  \int_{t_{m-1}}^{t_m} \frac{t_m - s}{\tau} \dd W_s =: W_a + W_b.
\end{align*}
Note, that $W_a$ and $W_b$ are independent. Moreover, $W_a$ is Gaussian with variance $\Sigma_a(u,v) = t_{m-1} \left(u,v \right)_U$ and due to Lemma~\ref{lem:inv-Gauss} $W_b$ is also Gaussian. Therefore $\mean{W}_m$ is Gaussian and it suffices to compute the mean and the variance operators.

Let $u,v \in U$. Then $\mathbb{E} \left[ \left( \mean{W}_m, u \right)_{U} \right] = 0$ and
\begin{align*}
&\mathbb{E} \left[\left( \mean{W}_m, u \right)_{U} \left( \mean{W}_m, v \right)_{U} \right] \\
&= \mathbb{E} \left[\left( W_a, u \right)_{U}\left( W_a, v \right)_{U} \right] + \mathbb{E} \left[\left( W_b, u\right)_{U}\left( W_b, v\right)_{U}  \right] \\
&= \left(u,v \right)_U \left( t_{m-1}  + \int_{t_{m-1}}^{t_m} \left(\frac{t_m - s}{\tau} \right)^2 \ds \right) \\
& = \left(u,v \right)_U  \frac{ 2t_{m-1} + t_m}{3}.
\end{align*}
The assertion is proved.
\end{proof}
\begin{corollary} \label{cor:dist-Bridge}
$\mean{\mathcal{B}_m}_m$ is an $U$-valued Gaussian random variable with zero mean and variance $\Sigma(u,v) = \frac{\tau}{12}(u,v)_U$.
\end{corollary}
\begin{proof}
The distribution of a random variable is uniquely determined by its characteristic function. Corollary~\ref{cor:Dist-meanW} implies
\begin{align*}
\varphi_{\mean{W}_m} (u) &= e^{- \frac{1}{2} \frac{2 t_{m-1} + t_m}{3} \norm{u}^2_U}.
\end{align*}
Classically, we find 
\begin{align*}
\varphi_{W(t_{m-1})}(u) &= e^{- \frac{1}{2} t_{m-1} \norm{u}^2_U}, \\
\varphi_{\frac{\Delta_m W}{2}}(u) &= e^{- \frac{1}{2} \frac{\tau}{4} \norm{u}^2_U}.
\end{align*}
The characteristic function of $\varphi_{\mean{W}_m}$ factors due to the independence of the decomposition~\eqref{def:Mean-decomp}, i.e.,
\begin{align*}
\varphi_{\mean{W}_m}(u) = \varphi_{W(t_{m-1}) + \mean{\mathcal{B}_m}_m + \frac{\Delta_m W}{2}}(u) = \varphi_{W(t_{m-1})}(u) \varphi_{\mean{\mathcal{B}_m}_m }(u) \varphi_{\frac{\Delta_m W}{2}}(u).
\end{align*}
Rearranging implies
\begin{align*}
\varphi_{\mean{\mathcal{B}_m}_m }(u) &= \frac{\varphi_{\mean{W}_m} (u)}{\varphi_{W(t_{m-1})}(u) \varphi_{\frac{\Delta_m W}{2}}(u)} \\
&= e^{- \frac{1}{2} \left(\frac{2 t_{m-1} + t_m}{3} - t_{m-1} - \frac{\tau}{4} \right) \norm{u}_U^2} = e^{- \frac{1}{2} \frac{\tau}{12}\norm{u}_U^2}.
\end{align*}
Overall, $\mean{\mathcal{B}_m}_m$ has the characteristic function of a Gaussian random variable with zero mean and variance $\Sigma(u,v) = \frac{\tau}{12}(u,v)_U$.
\end{proof}

At this point it is a simple task to find the distribution of the increments of the averaged Wiener process. Let us subtract~\eqref{def:Mean-decomp} for $m$ and $m-1$
\begin{align} \label{def:W-aver}
\Delta_m \mathbb{W} := \mean{W}_m - \mean{W}_{m-1} =  \frac{\Delta_m W + \Delta_{m-1} W}{2} + \mean{\mathcal{B}_m}_m - \mean{\mathcal{B}_{m-1}}_{m-1},
\end{align}
where we define $\mean{W}_0 := \Delta_0 W:= \mean{\mathcal{B}_0}_0:= 0$.
\begin{corollary} \label{cor:Distr-IncMean}
$\Delta_m \mathbb{W}$ is an $U$-valued Gaussian random variable with zero mean and variance $\Sigma(u,v) = \left( \frac{2\tau}{3} \chi_{\set{m \geq 2}} + \frac{\tau}{3} \chi_{\set{m=1}} \right) \left( u, v \right)_U$.
\end{corollary}
\begin{proof}
The right hand side of~\eqref{def:W-aver} is a sum of independent, centered Gaussian random variables. Thus the left hand side is centered Gaussian. Now, it suffices to compute the variance operator.

Note, in the case $m=1$ we have $\Delta_1 \mathbb{W} = \mean{W}_1$ and the result follows by Corollary~\ref{cor:Dist-meanW}. 

Let $m \geq 2$ and $u,v \in U$. Due to the independence,
\begin{align*}
&\mathbb{E} \left[\left( \Delta_m \mathbb{W}, u \right)_{U} \left(  \Delta_m \mathbb{W}, v \right)_{U} \right] \\
&\quad = \mathbb{E} \left[\left( \frac{\Delta_m W}{2}, u \right)_{U} \left(  \frac{\Delta_m W}{2}, v \right)_{U} \right] + \mathbb{E} \left[\left( \frac{\Delta_{m-1} W}{2}, u \right)_{U} \left(  \frac{\Delta_{m-1} W}{2}, v \right)_{U} \right] \\
&\quad \quad + \mathbb{E} \left[\left( \mean{\mathcal{B}_m}_m, u \right)_{U} \left( \mean{\mathcal{B}_m}_m, v \right)_{U} \right]  + \mathbb{E} \left[\left( \mean{\mathcal{B}_{m-1}}_{m-1}, u \right)_{U} \left( \mean{\mathcal{B}_{m-1}}_{m-1}, v \right)_{U} \right]  \\
&\quad = \frac{2\tau}{3}\left( u, v \right)_U.
\end{align*} 
\end{proof}
So far we have identified how each averaged increment~$\Delta_m \mathbb{W}$ is distributed. However, we also need to know what the joint distribution is, i.e., the distribution of a random vector.
\begin{lemma}\label{lem:distribution}
The random vector $(\Delta_m \mathbb{W})_{m=1}^M$ is an $U^M$-valued centered Gaussian random variable with variance operator $\Sigma :U^M \times U^M \to \R$ given by
\begin{align*}
 \Sigma (\bfu, \bfv) :=  \sum_{m,l=1}^M \sigma_{m,l} \left(u_m, v_l\right)_U,
\end{align*}
for $\bfu, \bfv \in U^M$ and 
\begin{align} \label{eq:sigma}
\sigma_{m,l} &= \begin{cases}
\frac{1}{3} \tau & \text{ if } l=m=1,\\
\frac{2}{3} \tau & \text{ if } l=m > 1, \\
\frac{1}{6} \tau & \text{ if }\abs{l-m} = 1, \\ 
0 & \text{ if }\abs{l-m}> 1.
\end{cases}
\end{align}
\end{lemma}
\begin{proof}
The equation~\eqref{def:W-aver} implies, that the random vector $(\Delta_m \mathbb{W})_{m=1}^M$ can be constructed via a linear transformation of independent Gaussian random vectors $(\Delta_m W)_{m=1}^M$ and $(\mean{\mathcal{B}_m}_m)_{m=1}^M$, i.e.,
\begin{align*}
\left( \Delta_m \mathbb{W} \right)_{m=1}^M = \mathbb{K}_1 (\Delta_m W)_{m=1}^M + \mathbb{K}_2 (\mean{\mathcal{B}_m}_m)_{m=1}^M,
\end{align*}
where $\mathbb{K}_1, \mathbb{K}_2 \in \R^{M \times M}$ are given by
\begin{align*}
\mathbb{K}_1 &= \frac{1}{2} 
\begin{pmatrix}
1 & 0 & 0 & 0 &\dots & 0 \\
1 & 1 & 0 & 0 &\dots & 0 \\
0 & 1 & 1 & 0 & \dots & 0 \\
\vdots & \vdots & \vdots & \vdots &  \vdots & \vdots \\
0 & 0 & \ldots & 1 & 1 & 0 \\ 
0 & 0 & \ldots & 0 & 1 & 1
\end{pmatrix}, \quad 
\mathbb{K}_2 = \begin{pmatrix}
1 & 0 & 0 & 0 &\dots & 0 \\
-1 & 1 & 0 & 0 &\dots & 0 \\
0 & -1 & 1 & 0 & \dots & 0 \\
\vdots & \vdots & \vdots & \vdots &  \vdots & \vdots \\
0 & 0 & \ldots & -1 & 1 & 0 \\ 
0 & 0 & \ldots & 0 & -1 & 1
\end{pmatrix}.
\end{align*}
Therefore, $(\Delta_m \mathbb{W})_{m=1}^M$ is itself a centered Gaussian vector. It remains to compute the covariance matrix. Let $u,v \in U$ and $m, l \in \set{1, \ldots, M}$. If $m = l$ Corollary~\ref{cor:Distr-IncMean} implies 
\begin{align*}
\mathbb{E} \left[ \left( \Delta_m \mathbb{W}, u \right)\left( \Delta_l \mathbb{W}, v\right) \right] = \left( \frac{2\tau}{3} \chi_{\set{m \geq 2}} + \frac{\tau}{3} \chi_{\set{m=1}} \right) \left( u, v \right)_U.
\end{align*}
If $\abs{m-l}> 1$, then equation~\eqref{def:W-aver} and the independence imply
\begin{align*}
\mathbb{E} \left[ \left( \Delta_m \mathbb{W}, u \right)\left( \Delta_l \mathbb{W}, v\right) \right] = 0.
\end{align*}
It remains to consider the case $\abs{m-l}=1$. Without loss of generality $l = m+1$. Now, using~\eqref{def:W-aver}, the independence and Corollary~\ref{cor:dist-Bridge}
\begin{align*}
&\mathbb{E} \left[ \left( \Delta_m \mathbb{W}, u \right)\left( \Delta_{m+1} \mathbb{W}, v\right) \right] \\
&\quad \quad = \mathbb{E} \left[ \left( \frac{\Delta_m W}{2}, u \right)\left( \frac{\Delta_m W}{2}, v\right) \right] - \mathbb{E} \left[ \left( \mean{\mathcal{B}_m}_m, u \right)\left( \mean{\mathcal{B}_m}_m, v\right) \right] \\
&\quad \quad = \left( \frac{\tau}{4} - \frac{\tau}{12} \right) \left(u,v \right)_U = \frac{\tau}{6} \left(u,v \right)_U.
\end{align*}
The proof is finished.
\end{proof}

\subsection{Sampling algorithm}
On the computer we are forced to approximate the continuous measure induced by the Wiener process $W$ by an empirical measure. Additionally, if we want to compare different numerical schemes we need to specify how to sample the random input needed for the involved algorithms jointly. More specifically, if we want to compare our algorithms~\eqref{eq:algo} and~\eqref{algo:2nd} and the classical Euler-Maruyama discretization~\eqref{eq:algo-EM}, we need to sample according to the law of the random vector
\begin{align} \label{eq:joint-W}
\left( \Delta_1 W, \ldots, \Delta_M W, \Delta_1 \mathbb{W}, \ldots, \Delta_M \mathbb{W} \right).
\end{align}
 
Based on the decomposition~\eqref{def:W-aver} we propose the following sampling algorithm.  The time discretization is achieved by the parameter $M \in \mathbb{N}$ and the series truncation in \eqref{rep:W} is done by the parameter $J \in \mathbb{N}$.

Given $M,J \in \mathbb{N}$.
\begin{enumerate}
\item (Sampling) Compute i.i.d. random variables $\zeta_{m}^j, \eta_{m}^j \sim \mathcal{N}(0,1)$ for $m \in \set{1,\ldots,M}$ and $j \in \set{1,\ldots,J}$.
\item (Lift to Hilbert space $U$) For $m \in \set{1,\ldots,M}$ define the random variables 
\begin{subequations} \label{algo:Sampling}
\begin{align} \label{eq:Z-new}
Z_m := \sqrt{\tau} \sum_{j=1}^J u^j \zeta_{m}^j, \quad  \tilde{\mathbb{Z}}_m := \sqrt{ \frac{\tau}{12} }\sum_{j=1}^J u^j \eta_{m}^j,
\end{align}
where $\set{u^j}_{j\in \mathbb{N}}$ is an orthonormal system of $U$.
\item (Adjusting correlation)  For $m \in \set{1,\ldots,M}$ define the random variables 
\begin{align} \label{eq:Z-new2}
\mathbb{Z}_m := \frac{Z_m + Z_{m-1}}{2} + \tilde{\mathbb{Z}}_m - \tilde{\mathbb{Z}}_{m-1},
\end{align}
\end{subequations}
where $Z_0 := \tilde{\mathbb{Z}}_0 := 0$.
\end{enumerate} 
%
%
%

Let $\Pi_J$ be the $U$-orthogonal projection onto $U_J := \text{span}(u_1,\ldots, u_J)$. The following proposition guarantees that the sampling algorithm~\eqref{algo:Sampling} approximates the desired random variables.

\begin{proposition} \label{prop:joint_law}
Let $\bfZ := \left( Z_1, \ldots, Z_M,\mathbb{Z}_1,\ldots,\mathbb{Z}_M \right) \in U^{2M}$ be generated by \eqref{eq:Z-new} and \eqref{eq:Z-new2}. Then,
\begin{align*}
\bfZ  \sim\left( \Pi_J \Delta_1 W, \ldots, \Pi_J \Delta_M W, \Pi_J \Delta_1 \mathbb{W}, \ldots, \Pi_J \Delta_M \mathbb{W} \right),
\end{align*}
where $\Pi_J$ is the $U$-orthogonal projection onto $\text{span}(u^1,\ldots, u^J)$.
\end{proposition}
\begin{proof}
First, we need to observe that $Z_m \sim \Pi_J \Delta_m W$ and $\tilde{Z}_m \sim \Pi_J \mean{\mathcal{B}_m }_m$. Then, the statement follows similarly to the proof of Lemma~\ref{lem:distribution}. 
\end{proof}

\begin{remark}
Proposition~\ref{prop:joint_law} ensures that we can compare our algorithm to the classical Euler-Maruyama discretization~\eqref{eq:algo-EM} on an equidistant time grid. It can be adjusted to also match non-equidistant grids.
\end{remark}

\section{Simulations} \label{sec:Simulations}
In this section we perform numerical simulations to test our algorithms~\eqref{eq:algo} and~\eqref{algo:2nd}. We denote the solution of~\eqref{eq:algo} as $\bfv^{\text{Half}}$ and the solution of~\eqref{algo:2nd} as $\bfv^{\text{Full}}$. Additionally, we compare our algorithms to the classical Euler-Maruyama discretization of~\eqref{eq:p-Laplace}. That reads, find $\bfv \in (V_h)^{M+1}$ such that for all $\xi_h \in V_h$, $m \geq 1$ and $\mathbb{P}$-a.s.
\begin{align} \label{eq:algo-EM}
\left(v_m- v_{m-1}, \xi_h \right) + \tau \left( S(\nabla v_m), \nabla \xi_h \right) = \left( G(v_{m-1}) \Delta_m W, \xi_h \right).
\end{align}
The solution to~\eqref{eq:algo-EM} is called $\bfv^{\text{EM}}$. So far it is unknown whether the approximation $\bfv^{\text{EM}}$ converges to the solution $u$ of~\eqref{eq:p-Laplace} in a suitable sense if the regularity of the map $t \mapsto V(\nabla u_t)$ is limited. Originally the Euler-Maruyama scheme has been introduced for stochastic differential equations. In this context much more is known and convergence has been proven, see e.g. the book of Kloeden and Platen \cite[Section~$9.5$]{MR1214374}.  However, when dropping the Lipschitz assumption on the coefficients, divergence with positive probability has been obtained in \cite{MR2795791}.

We are particularly interested in the experimental study of the following questions:
\begin{enumerate}
\item Do the algorithms~\eqref{eq:algo} and~\eqref{algo:2nd} approximate mean-values or point-values?
\item How does the Euler-Maruyama scheme~\eqref{eq:algo-EM} compares to~\eqref{eq:algo} and~\eqref{algo:2nd} in terms of time and space convergence?
\item How do the different error quantities~\eqref{eq:Rates},~\eqref{eq:aver_conv1} and~\eqref{eq:aver_conv2} for the gradient relate to each other?
\item How sensitive are the algorithms with respect to the parameter $p$?
\end{enumerate}

 All simulations are done with the help of the open source tool for solving partial differential equations FEniCS~\cite{LoggMardalWells12}.

\subsection{An explicit solution}
In the linear case, $p=2$, with a linear right hand side
\begin{align}
G(v) \Delta W = \lambda v \, \Delta \beta^1,
\end{align}
for some $\lambda \in \R$, it is possible to find an explicit solution. If we start the evolution defined by~\eqref{eq:p-Laplace} in an eigenfunction of the Laplace operator, the dynamics becomes simpler. Let $\mathcal{O} = (0,1)^2$, $T = 1$ and $u_0(x) = \sin(\pi x_1) \sin( \pi x_2)$. Note that $u_0$ is an eigenfunction of the $2$-Laplacian with homogeneous Dirichlet data and corresponding eigenvalue $\mu = 2 \pi^2$. The unique solution to~\eqref{eq:p-Laplace} is given by
\begin{align}
\begin{aligned}
u(\omega,t,x) &= \exp \left\{ - \left( \frac{\lambda^2}{2}+  \mu \right)t + \lambda \beta^1(\omega,t)\right\} u_0(x).
\end{aligned}
\end{align}
Similarly, we can give a closed expression for the solution $u_h$ to the space-discrete equation
\begin{align} \label{eq:Laplace-semidisc}
\dd  \left( u_h, \xi_h \right) + \left(\nabla u_h, \nabla \xi_h \right)  \dt = \lambda \left(u_h,\xi_h \right) \dd \beta^1(t).
\end{align}
This is equivalent to the system of linear stochastic differential equations
\begin{align} \label{eq:Laplace-semidiscMat}
\dd M \bfu + S \bfu \dt = \lambda M \bfu \dd \beta^1(t),
\end{align}
where $M_{i,j} = \left( \xi_h^j, \xi_h^i \right)$ and $S_{i,j} = \left( \nabla \xi_h^j, \nabla \xi_h^i \right)$ is the mass respectively the stiffness matrix and $\set{\xi_h^i}$ form a basis of $V_{h}$. The eigenpairs of the discretized Laplacian are related to the linear system
\begin{align} \label{eq:Eigenvalue}
S \bfu = \mu M \bfu \Leftrightarrow M^{-1} S \bfu = \mu \bfu.
\end{align}
Let $(\mu_h,\bfu_h) \in (0,\infty) \times \R^{\abs{V_{h}}}$ be a solution to~\eqref{eq:Eigenvalue}, then the solution to~\eqref{eq:Laplace-semidisc} started in $u_h(0) = \bfu_h \cdot \bfxi_h$ is given by
\begin{align} \label{eq:SemiDisSolution}
u_h(\omega,t,x) = \exp \left\{ -\left( \frac{\lambda^2}{2} + \mu_h \right)t + \lambda \beta^1(\omega,t) \right\} u_h(0).
\end{align}
The main advantage of having an analytic solution of the space discrete equation is, that it rules out any space discretization errors. 

To accurately compare continuous processes and discrete vectors, one needs to either lift the vector to a process or project the process to a vector. We do the latter approach and evaluate the continuous process $u_h$ on the equidistant partition of $I$. This leads to the definition 
\begin{align}\label{def:Point}
\bfu^{\text{Point}} := \left( u_h(t_m) \right)_{m=1}^M.
\end{align}
Since the algorithms~\eqref{eq:algo} and~\eqref{algo:2nd} approximate the mean value of the analytic solution, we define 
\begin{align} \label{def:Aver-exact}
\bfu_{\text{exact}}^{\text{Aver}} := \left( \mean{u_h}_{I_m} \right)_{m=1}^M.
\end{align}

Although we know the exact solution, the time averages of the exact solution are non-treatable without knowledge of the full trajectory of the Brownian motion $\beta^1$. Numerically, we substitute $\mean{u_h}_{I_m}$ by a Riemann sum approximation, i.e. we fix an equidistant partition $\set{[t_{m,k-1},t_{m,k}]}_{k=1}^r$ of $I_m$ with resolution $r \in \mathbb{N}$ and define
\begin{align} \label{def:Aver-approx}
\bfu^{\text{Aver}} := \left( \frac{1}{r} \sum_{k=1}^r u_h(t_{m,k}) \right)_{m=1}^M.
\end{align}
The approximation quality is measured in the error quantities
\begin{subequations} \label{error:glob}
\begin{align} \label{error:1st}
\mathcal{E}( \bfu,\bfv ) &:= \mathbb{E} \left[ \max_{m=1,\ldots,M} \norm{u_m - v_m}_{L^2_x}^2 \right], \\ \label{error:2nd}
\mathcal{V}( \bfu, \bfv ) &:= \mathbb{E} \left[ \sum_{m=1}^M \tau \norm{\nabla ( u_m -  v_m) }_{L^2_x}^2 \right],
\end{align}
\end{subequations}
where $\bfu \in \set{ \bfu^{\text{Point}}, \bfu^{\text{Aver}}}$ and $\bfv \in \set{\bfv^{\text{EM}},\bfv^{\text{Full}},\bfv^{\text{Half}} } $.

In Figure~\ref{fig:SemiDiscrete} we plot the time convergence of the error quantities~\eqref{error:1st} and~\eqref{error:2nd}. We approximate the expectation by the Monte-Carlo method with $20$ samples. Additionally, we let $\lambda = 1$ and $V_h$ to be the space of piecewise linear, continuous elements on a uniform mesh with $\abs{V_h} = 121$. The average values~\eqref{def:Aver-approx} are approximated by $r = 10$.

The numerical results support that $\bfv^{\text{EM}}$ approximates the solution on the grid points, while $\bfv^{\text{Full}}$ and $\bfv^{\text{Half}}$ approximate the average values of the solution. The gap is due to the difference
\begin{align*}
\mathbb{E} \left[ \norm{\mean{u}_{I_m} - u(t_m)}_{L^2_x}^2 \right] \geq c \tau \norm{u_0}_{L^2_x}^2.
\end{align*} 
Initially, we observe a preasymptotic effect. It stabilizes at the time scale $\tau \approx 10^{-3}$. Afterwards the predicted convergence speed of order~$1$ is achieved.
%


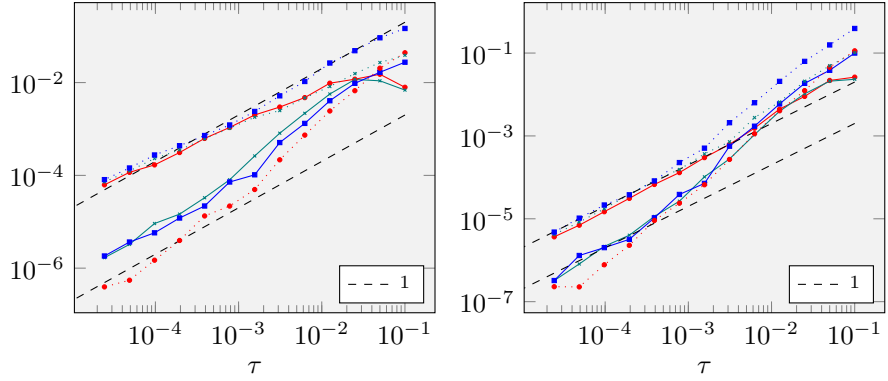
\begin{figure}
\begin{center}
\begin{tikzpicture}
\begin{axis}[
clip=false,
width=.5\textwidth,
height=.45\textwidth,
xmode = log,
ymode = log,
cycle multi list={\nextlist MyColors},
xlabel={$\tau$},
scale = {1},
clip = true,
legend cell align=left,
legend style={legend columns=1,legend pos= south east,font=\fontsize{7}{5}\selectfont}
]
	\addplot table [x=dt,y=LinftyL2aver] {Data/p=2 explicit solution/averClassic_Expmult_p_2_hmax_0.14142135623730964_dtfine_2.44140625e-05_glinCrop.txt};
	\addplot table [x=dt,y=LinftyL2aver] {Data/p=2 explicit solution/averFull_Expmult_p_2_hmax_0.14142135623730964_dtfine_2.44140625e-05_glinCrop.txt};
	\addplot table [x=dt,y=LinftyL2aver] {Data/p=2 explicit solution/averHalf_Expmult_p_2_hmax_0.14142135623730964_dtfine_2.44140625e-05_glinCrop.txt};
	
	\addplot table [x=dt,y=LinftyL2right] {Data/p=2 explicit solution/averClassic_Expmult_p_2_hmax_0.14142135623730964_dtfine_2.44140625e-05_glinCrop.txt};
	\addplot table [x=dt,y=LinftyL2right] {Data/p=2 explicit solution/averFull_Expmult_p_2_hmax_0.14142135623730964_dtfine_2.44140625e-05_glinCrop.txt};
	\addplot table [x=dt,y=LinftyL2right] {Data/p=2 explicit solution/averHalf_Expmult_p_2_hmax_0.14142135623730964_dtfine_2.44140625e-05_glinCrop.txt};

\addplot[dashed,sharp plot,update limits=false] coordinates {(1e-1,2e-1) (1e-6,2e-6)};
\addplot[dashed,sharp plot,update limits=false] coordinates {(1e-1,2e-3) (1e-6,2e-8)};

	\legend{
	,,,,,,
	{$1$},
	};
\end{axis}
\end{tikzpicture}
\begin{tikzpicture}
\begin{axis}[
clip=false,
width=.5\textwidth,
height=.45\textwidth,
xmode = log,
ymode = log,
cycle multi list={\nextlist MyColors},
xlabel={$\tau$},
scale = {1},
clip = true,
legend cell align=left,
legend style={legend columns=1,legend pos= south east,font=\fontsize{7}{5}\selectfont}
]
		\addplot table [x=dt,y=L2Vaver] {Data/p=2 explicit solution/averClassic_Expmult_p_2_hmax_0.14142135623730964_dtfine_2.44140625e-05_glinCrop.txt};
	\addplot table [x=dt,y=L2Vaver] {Data/p=2 explicit solution/averFull_Expmult_p_2_hmax_0.14142135623730964_dtfine_2.44140625e-05_glinCrop.txt};
	\addplot table [x=dt,y=L2Vaver] {Data/p=2 explicit solution/averHalf_Expmult_p_2_hmax_0.14142135623730964_dtfine_2.44140625e-05_glinCrop.txt};
	
	\addplot table [x=dt,y=L2Vright] {Data/p=2 explicit solution/averClassic_Expmult_p_2_hmax_0.14142135623730964_dtfine_2.44140625e-05_glinCrop.txt};
	\addplot table [x=dt,y=L2Vright] {Data/p=2 explicit solution/averFull_Expmult_p_2_hmax_0.14142135623730964_dtfine_2.44140625e-05_glinCrop.txt};
	\addplot table [x=dt,y=L2Vright] {Data/p=2 explicit solution/averHalf_Expmult_p_2_hmax_0.14142135623730964_dtfine_2.44140625e-05_glinCrop.txt};

\addplot[dashed,sharp plot,update limits=false] coordinates {(1e-1,2e-2) (1e-6,2e-7)};
\addplot[dashed,sharp plot,update limits=false] coordinates {(1e-1,2e-3) (1e-6,2e-8)};

	\legend{
	,,,,,,
	{$1$},
	};
\end{axis}
\end{tikzpicture}
\caption{Time convergence of $\bfv^{\text{EM}}$~(red), $\bfv^{\text{Half}}$~(blue) and $\bfv^{\text{Full}}$~(green) towards $\bfu^{\text{Point}}$~(dashed) respectively $\bfu^{\text{Aver}}$~(solid) measured in the error terms $\mathcal{E}$~(left) and $\mathcal{V}$~(right). }
\label{fig:SemiDiscrete}
\end{center}
\end{figure}

\subsection{Beyond known solutions}
In general, a major obstacle is the absence of an analytic solution to the equation~\eqref{eq:p-Laplace}. In particular, the distance of the numerical solution and the analytic solution as presented in~\eqref{eq:Rates}, \eqref{eq:Rates2} respectively~\eqref{eq:Rates-algo-2nd} and~\eqref{eq:Rates-algo-2nd-2} are non-computable. 

To overcome this difficulty we measure the error of a fine reference approximation $\bfv_f$ and a coarse approximation $\bfv_c$. Both, $\bfv_f$ and $\bfv_c$, are generated via the same algorithm (either~\eqref{eq:algo},~\eqref{algo:2nd} or~\eqref{eq:algo-EM}) on a fine respectively coarse scale. 
Let $h_c \geq h_f >0$ be coarse respectively fine space mesh sizes. Similarly, let $M_c, M_f \in \mathbb{N}$ be coarse respectively fine time discretization parameters with corresponding timestep sizes $\tau_c$, $\tau_f$. For simplicity, we assume $M_c /M_f = r \in \mathbb{N}$. Now, the coarse intervals are generated by the fine ones, i.e.
\begin{align*}
I_m^c &:= [(m-1)\tau_c,m\tau_c] = \bigcup_{k=1}^r I_{(m-1)r + k}^f.
\end{align*}
The coarse averaging operator can be decomposed into the fine averaging operator 
\begin{align*}
\mean{u}_{I_m^c} = \dashint_{I_m^c} u_\nu \dd \nu =  \frac{1}{r \tau_f} \sum_{k=1}^r \int_{I_{(m-1)r + k}^f} u_\nu \dd \nu  = \frac{1}{r} \sum_{k=1}^r \mean{u}_{I_{(m-1)r + k}^f}.
\end{align*}
To accurately substitute the analytic averaging operator, we define the discrete time averaging operator
\begin{align*}
\mean{\bfv_f}_{m}^r:=  \frac{1}{r} \sum_{k=1}^r v_{(m-1)r + k}^f.
\end{align*}
Additionally, we define the error quantities
\begin{subequations} \label{LinfL2-discrete}
\begin{align}
d_{L^\infty L^2}^{\,\text{aver}}( \bfv_f, \bfv_c) &:= \mathbb{E} \left[ \max_{m=1,\ldots, M_c} \norm{ \mean{\bfv_f}_{m}^r - v_m^c }_{L^2_x}^2 \right], \\
 d_{L^\infty L^2}^{\,\text{point}}( \bfv_f, \bfv_c) &:= \mathbb{E} \left[ \max_{m=1,\ldots, M_c} \norm{ v_{mr}^f - v_m^c }_{L^2_x}^2 \right],
\end{align}
\end{subequations}
and
\begin{subequations}\label{L2V-discrete}
\begin{align}
d_{L^2 V}^{\,\text{classic}}( \bfv_f, \bfv_c) &:= \mathbb{E} \left[ \sum_{m=1}^{M_c} \frac{\tau_c}{r} \sum_{k=1}^r \norm{V(\nabla v^f_{(m-1)r + k}) - V(\nabla v^c_m )}_{L^2_x}^2 \right], \\
d_{L^2 V}^{\,\text{inner}}( \bfv_f, \bfv_c) &:= \mathbb{E} \left[ \sum_{m=1}^{M_c} \tau_c \norm{V(\nabla \mean{\bfv_f}_{m}^r) - V(\nabla v^c_m )}_{L^2_x}^2 \right], \\
d_{L^2 V}^{\,\text{outer}}( \bfv_f, \bfv_c) &:= \mathbb{E} \left[ \sum_{m=1}^{M_c} \tau_c \norm{\mean{ V(\nabla \bfv_f)}_m^r - V(\nabla v^c_m )}_{L^2_x}^2 \right].
\end{align}
\end{subequations}

\subsection{Joint sampling on fine and coarse scales}
It is crucial to use the same stochastic input when computing $\bfv_f$ and $\bfv_c$. This can be done in two different ways. Either one first samples coarse stochastic data and a posteriori samples the fine data based on the conditional probabilities of the coarse one, or we can sample the fine stochastic input and try to reconstruct the coarse stochastic data. The latter approach is more suitable for the averaged increments. 
\begin{lemma} \label{lem:Coarse-reconstruction}
Let $M_f, M_c \in \mathbb{N}$. Assume $M_f = r M_c$ for some $r \in \mathbb{N}$. Then
\begin{subequations} \label{eq:Rel-Coarse-fine-full}
\begin{align}\label{eq:Rel-Coarse-fine}
\Delta_1^c \mathbb{W} &= \sum_{l=1}^r \left(1- \frac{l-1}{r} \right) \Delta_l^f \mathbb{W}, \\
\Delta^c_j \mathbb{W} &=  \sum_{l=0}^{r-1} \frac{l+1}{r} \Delta_{rj - l}^f \mathbb{W} +\sum_{l=0}^{r-2} \left(1- \frac{l+1}{r }\right) \Delta_{r(j-1) -l}^f \mathbb{W},
\end{align}
for $j \in \set{2,\ldots, M_c}$.
\end{subequations}
\end{lemma}
\begin{proof}
Since $M_f = r M_c$, we have $\tau_c = r \tau_f$. Thus,
\begin{align*}
\Delta_1^c \mathbb{W} &= \frac{1}{\tau_c} \int_{0}^{\tau_c} W_s \ds = \frac{1}{r \tau_f} \int_{0 \tau_f}^{r \tau_f} W_s \ds = \frac{1}{r \tau_f} \sum_{l=1}^r \int_{(l-1)\tau_f}^{l\tau_f} W_s \ds \\
&= \frac{1}{r \tau_f} \sum_{l=1}^r \int_{(l-1)\tau_f}^{l\tau_f} W_s - W_{s-(l-1)\tau_f} \ds +  \frac{1}{r \tau_f} \sum_{l=1}^r \int_{(l-1)\tau_f}^{l\tau_f} W_{s-(l-1)\tau_f} \ds \\
&=: \mathrm{I} + \mathrm{II}.
\end{align*}
Due to the discrete Fubini's theorem
\begin{align*}
\mathrm{I} &= \frac{1}{r \tau_f} \sum_{l=1}^r \int_{(l-1)\tau_f}^{l\tau_f} \sum_{l'=0}^{l-2} W_{s - l' \tau_f} - W_{s-(l'+1)\tau_f} \ds \\
&= \frac{1}{r \tau_f} \sum_{l=1}^r \sum_{l'=0}^{l-2} \int_{(l-l'-1)\tau_f}^{(l-l')\tau_f}  W_{s } - W_{s-\tau_f} \ds\\
&= \frac{1}{r} \sum_{l=1}^r \sum_{l'=0}^{l-2} \Delta_{l-l'}^f \mathbb{W} = \sum_{l=2}^r \left(1+ \frac{1-l}{r} \right) \Delta_l^f \mathbb{W}.
\end{align*}
The second term is easily computed
\begin{align*}
\mathrm{II} = \frac{1}{r \tau_f} \sum_{l=1}^r \int_{0}^{\tau_f} W_{s} \ds = \Delta_1^f \mathbb{W}.
\end{align*}
Overall,
\begin{align*}
\Delta_1^c \mathbb{W} = \sum_{l=1}^r \left(1+ \frac{1-l}{r} \right) \Delta_l^f \mathbb{W}.
\end{align*}
Let $j \in \set{2,\ldots,M_c}$. Since $t_{j}^c = t_{rj}^f$, it holds
\begin{align*}
\Delta^c_j \mathbb{W} &= \frac{1}{\tau_c} \int_{t_{j-1}^c}^{t_j^c} W_s - W_{s-\tau_c} \ds = \frac{1}{r \tau_f} \int_{t_{r(j-1)}^f}^{t_{rj}^f} W_s - W_{s-r\tau_f} \ds \\
&= \frac{1}{r \tau_f} \sum_{l=0}^{r-1}  \int_{t_{rj - (l + 1)}^f}^{t_{r j - l}^f} \sum_{l' = 0 }^{r-1} W_{s-l' \tau_f} - W_{s-(l'+1) \tau_f} \ds \\
&= \frac{1}{r \tau_f} \sum_{l=0}^{r-1} \sum_{l' = 0 }^{r-1} \int_{t_{rj - (l + l' + 1)}^f}^{t_{r j - (l + l') }^f}  W_{s} - W_{s- \tau_f} \ds = \frac{1}{r} \sum_{l=0}^{r-1} \sum_{l' = 0 }^{r-1} \Delta_{rj - (l+l')}^f \mathbb{W}.
\end{align*}
The discrete Fubini's theorem implies
\begin{align*}
\sum_{l=0}^{r-1} \sum_{l' = 0 }^{r-1} \Delta_{rj - (l+l')}^f \mathbb{W} = \sum_{l=0}^{r-1} (l+1) \Delta_{rj - l}^f \mathbb{W} +\sum_{l=0}^{r-2} (r-(l+1)) \Delta_{r(j-1) -l}^f \mathbb{W}.
\end{align*}
Therefore,
\begin{align*}
\Delta^c_j \mathbb{W} =  \sum_{l=0}^{r-1} \frac{l+1}{r} \Delta_{rj - l}^f \mathbb{W} +\sum_{l=0}^{r-2} \left(1- \frac{l+1}{r }\right) \Delta_{r(j-1) -l}^f \mathbb{W}.
\end{align*}
The claim is proved.
\end{proof}

\begin{remark}
The reconstruction formula \eqref{eq:Rel-Coarse-fine-full} is the key ingredient, why it is possible to compare fine and coarse numerical solutions in an efficient way. If one tries to establish a corresponding formula for randomized algorithms as proposed in \cite{MR4298537}, this task becomes more challenging.
\end{remark}

\subsection{Simulation: Unknown solution}
Let $\mathcal{O} = (0,1)^2$ and $T = 1$. We choose $p \in \set{1.5,3}$, $u_0(x) = \sin(\pi x_1) \sin(\pi x_2)$ and
\begin{align} \label{eq:G-first}
G(u)\Delta W &:= \underbrace{\sin(\pi x_1)x_2  u}_{=: g_1(x,u)} \Delta \beta^1 + \underbrace{\sin(\pi x_2)x_1 u}_{=: g_2(x,u)} \Delta \beta^2.
\end{align}
$V_h$ denotes the space of piece wise linear elements with zero boundary values on a triangulation $\mathcal{T}_h$ of $\mathcal{O}$. We initialize the coarse triangulation $\mathcal{T}_{h_c}$ as a uniform triangulation of $\mathcal{O}$ such that $\abs{V_{h_c}} = 121$ and $h_c \approx 1.4*10^{-1}$. $\mathcal{T}_{h_f}$ is generated by three uniform refinements of $\mathcal{T}_{h_c}$. Then $\abs{V_{h_f}} = 6561$ and $h_f \approx 1.7*10^{-2}$. For the time discretization we use $M_f = 1280$ and $M_c = 40$. Therefore $\tau_f \approx 7.8*10^{-4}$ and $\tau_c \approx 2.5* 10^{-2}$. We measure the error of the fine numerical solution $\bfv_f \in \set{\bfv_f^{\text{EM}},\bfv_f^{\text{Half}},\bfv_f^{\text{Full}}}$ versus the coarse numerical solution $\bfv_c \in \set{\bfv_c^{\text{EM}},\bfv_c^{\text{Half}},\bfv_c^{\text{Full}}} $ of the same algorithm in the error quantities~\eqref{LinfL2-discrete} and~\eqref{L2V-discrete}. The expectation is approximated by the Monte-Carlo method with $20$ samples. 

In Figure~\ref{fig:Unknown-p=3} respectively Figure~\ref{fig:Unknown-p=1_5} we plot the evolution of the error for $p=3$ respectively $p = 1.5$. In both cases we observe linear convergence. This indicates that on the used discretization scale the space error dominates the time error. We do not see a substantial difference in the gradient error quantities. Additionally, $\bfv^{\text{EM}}$ measured in the point distance $d_{L^\infty L^2}^{\,\text{point}}$ and $\bfv^{\text{Half}}$ measured in the averaged distance $d_{L^\infty L^2}^{\,\text{aver}}$ perform equally well. 

If $p=3$ then $\bfv^{\text{Full}}$ behaves similarly in both error terms and performs slightly worse than $\bfv^{\text{EM}}$ and $\bfv^{\text{Half}}$. Contrary in the case $p=1.5$, while still performing slightly worse compared to $\bfv^{\text{EM}}$ and $\bfv^{\text{Half}}$, $\bfv^{\text{Full}}$ is better approximated in $d_{L^\infty L^2}^{\,\text{aver}}$ than in $d_{L^\infty L^2}^{\,\text{point}}$. This indicates that at least in the singular setting $\bfv^{\text{Full}}$ also approximates average values.

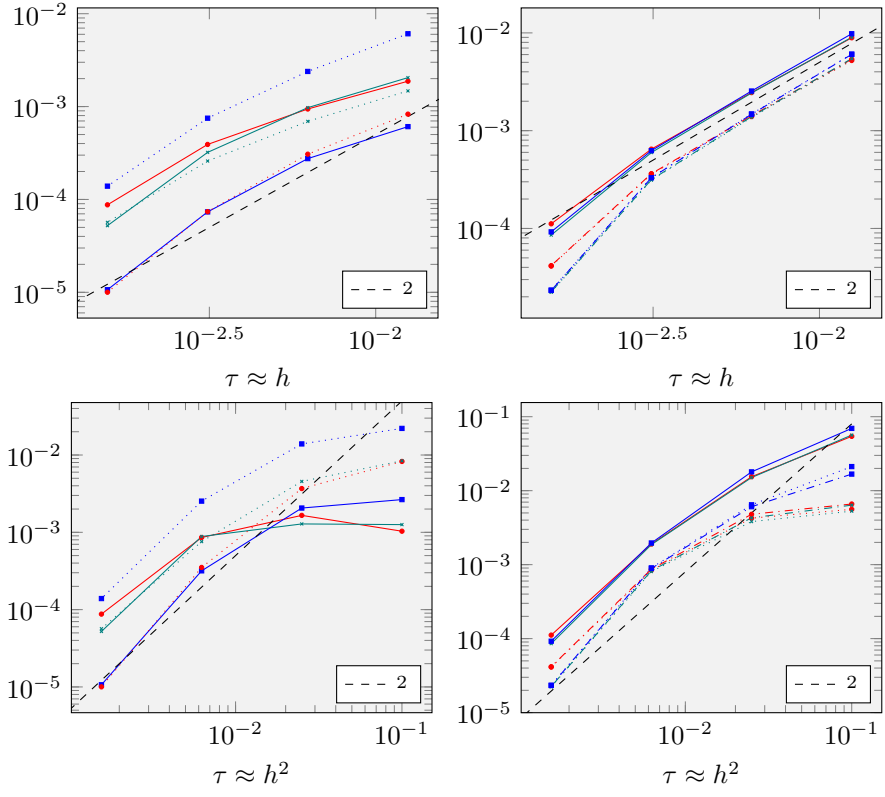
\begin{figure}
\begin{center}
\begin{tikzpicture}
\begin{axis}[
clip=false,
width=.5\textwidth,
height=.45\textwidth,
xmode = log,
ymode = log,
cycle multi list={\nextlist MyColors},
xlabel={$\tau \approx h$},
scale = {1},
clip = true,
legend cell align=left,
legend style={legend columns=1,legend pos= south east,font=\fontsize{7}{5}\selectfont}
]
	\addplot table [x=dt,y=LinftyL2aver] {Data/p=3 unknown/Classic/h=tau/averClassic_Expmult_p_3_hmax_0.01767766952966378_dtfine_0.00078125_gScrop.txt};
	\addplot table [x=dt,y=LinftyL2aver] {Data/p=3 unknown/Full/h=tau/averFull_Expmult_p_3_hmax_0.01767766952966378_dtfine_0.00078125_gScrop.txt};
	\addplot table [x=dt,y=LinftyL2aver] {Data/p=3 unknown/Half/h=tau/averHalf_Expmult_p_3_hmax_0.01767766952966378_dtfine_0.00078125_gScrop.txt};
	
	\addplot table [x=dt,y=LinftyL2clas] {Data/p=3 unknown/Classic/h=tau/averClassic_Expmult_p_3_hmax_0.01767766952966378_dtfine_0.00078125_gScrop.txt};
	\addplot table [x=dt,y=LinftyL2clas] {Data/p=3 unknown/Full/h=tau/averFull_Expmult_p_3_hmax_0.01767766952966378_dtfine_0.00078125_gScrop.txt};
	\addplot table [x=dt,y=LinftyL2clas] {Data/p=3 unknown/Half/h=tau/averHalf_Expmult_p_3_hmax_0.01767766952966378_dtfine_0.00078125_gScrop.txt};
	
\addplot[dashed,sharp plot,update limits=false] coordinates {(1e-1,5e-2) (1e-3,5e-6)};

	\legend{
	,,,,,,
	{$2$},
	};
\end{axis}
\end{tikzpicture}
\begin{tikzpicture}
\begin{axis}[
clip=false,
width=.5\textwidth,
height=.45\textwidth,
xmode = log,
ymode = log,
cycle multi list={\nextlist MyColors},
xlabel={$\tau \approx h$},
scale = {1},
clip = true,
legend cell align=left,
legend style={legend columns=1,legend pos= south east,font=\fontsize{7}{5}\selectfont}
]
		\addplot table [x=dt,y=L2V1] {Data/p=3 unknown/Classic/h=tau/averClassic_Expmult_p_3_hmax_0.01767766952966378_dtfine_0.00078125_gScrop.txt};
	\addplot table [x=dt,y=L2V1] {Data/p=3 unknown/Full/h=tau/averFull_Expmult_p_3_hmax_0.01767766952966378_dtfine_0.00078125_gScrop.txt};
	\addplot table [x=dt,y=L2V1] {Data/p=3 unknown/Half/h=tau/averHalf_Expmult_p_3_hmax_0.01767766952966378_dtfine_0.00078125_gScrop.txt};
	
	\addplot table [x=dt,y=L2V2] {Data/p=3 unknown/Classic/h=tau/averClassic_Expmult_p_3_hmax_0.01767766952966378_dtfine_0.00078125_gScrop.txt};
	\addplot table [x=dt,y=L2V2] {Data/p=3 unknown/Full/h=tau/averFull_Expmult_p_3_hmax_0.01767766952966378_dtfine_0.00078125_gScrop.txt};
	\addplot table [x=dt,y=L2V2] {Data/p=3 unknown/Half/h=tau/averHalf_Expmult_p_3_hmax_0.01767766952966378_dtfine_0.00078125_gScrop.txt};
	
		\addplot table [x=dt,y=L2V3] {Data/p=3 unknown/Classic/h=tau/averClassic_Expmult_p_3_hmax_0.01767766952966378_dtfine_0.00078125_gScrop.txt};
	\addplot table [x=dt,y=L2V3] {Data/p=3 unknown/Full/h=tau/averFull_Expmult_p_3_hmax_0.01767766952966378_dtfine_0.00078125_gScrop.txt};
	\addplot table [x=dt,y=L2V3] {Data/p=3 unknown/Half/h=tau/averHalf_Expmult_p_3_hmax_0.01767766952966378_dtfine_0.00078125_gScrop.txt};
	
\addplot[dashed,sharp plot,update limits=false] coordinates {(1e-1,5e-1) (1e-3,5e-5)};

	\legend{
	,,,,,,,,,
	{$2$},
	};
\end{axis}
\end{tikzpicture}
\begin{tikzpicture}
\begin{axis}[
clip=false,
width=.5\textwidth,
height=.45\textwidth,
xmode = log,
ymode = log,
cycle multi list={\nextlist MyColors},
xlabel={$\tau \approx h^2$},
scale = {1},
clip = true,
legend cell align=left,
legend style={legend columns=1,legend pos= south east,font=\fontsize{7}{5}\selectfont}
]
	\addplot table [x=dt,y=LinftyL2aver] {Data/p=3 unknown/Classic/h=sqrt(tau)/averClassic_Expmult_p_3_hmax_0.01767766952966378_dtfine_0.00078125_gScrop.txt};
	\addplot table [x=dt,y=LinftyL2aver] {Data/p=3 unknown/Full/h=sqrt(tau)/averFull_Expmult_p_3_hmax_0.01767766952966378_dtfine_0.00078125_gScrop.txt};
	\addplot table [x=dt,y=LinftyL2aver] {Data/p=3 unknown/Half/h=sqrt(tau)/averHalf_Expmult_p_3_hmax_0.01767766952966378_dtfine_0.00078125_gScrop.txt};
	
	\addplot table [x=dt,y=LinftyL2clas] {Data/p=3 unknown/Classic/h=sqrt(tau)/averClassic_Expmult_p_3_hmax_0.01767766952966378_dtfine_0.00078125_gScrop.txt};
	\addplot table [x=dt,y=LinftyL2clas] {Data/p=3 unknown/Full/h=sqrt(tau)/averFull_Expmult_p_3_hmax_0.01767766952966378_dtfine_0.00078125_gScrop.txt};
	\addplot table [x=dt,y=LinftyL2clas] {Data/p=3 unknown/Half/h=sqrt(tau)/averHalf_Expmult_p_3_hmax_0.01767766952966378_dtfine_0.00078125_gScrop.txt};
	
\addplot[dashed,sharp plot,update limits=false] coordinates {(1e-1,5e-2) (1e-3,5e-6)};

	\legend{
	,,,,,,
	{$2$},
	};
\end{axis}
\end{tikzpicture}
\begin{tikzpicture}
\begin{axis}[
clip=false,
width=.5\textwidth,
height=.45\textwidth,
xmode = log,
ymode = log,
cycle multi list={\nextlist MyColors},
xlabel={$\tau \approx h^2$},
scale = {1},
clip = true,
legend cell align=left,
legend style={legend columns=1,legend pos= south east,font=\fontsize{7}{5}\selectfont}
]
		\addplot table [x=dt,y=L2V1] {Data/p=3 unknown/Classic/h=sqrt(tau)/averClassic_Expmult_p_3_hmax_0.01767766952966378_dtfine_0.00078125_gScrop.txt};
	\addplot table [x=dt,y=L2V1] {Data/p=3 unknown/Full/h=sqrt(tau)/averFull_Expmult_p_3_hmax_0.01767766952966378_dtfine_0.00078125_gScrop.txt};
	\addplot table [x=dt,y=L2V1] {Data/p=3 unknown/Half/h=sqrt(tau)/averHalf_Expmult_p_3_hmax_0.01767766952966378_dtfine_0.00078125_gScrop.txt};
	
	\addplot table [x=dt,y=L2V2] {Data/p=3 unknown/Classic/h=sqrt(tau)/averClassic_Expmult_p_3_hmax_0.01767766952966378_dtfine_0.00078125_gScrop.txt};
	\addplot table [x=dt,y=L2V2] {Data/p=3 unknown/Full/h=sqrt(tau)/averFull_Expmult_p_3_hmax_0.01767766952966378_dtfine_0.00078125_gScrop.txt};
	\addplot table [x=dt,y=L2V2] {Data/p=3 unknown/Half/h=sqrt(tau)/averHalf_Expmult_p_3_hmax_0.01767766952966378_dtfine_0.00078125_gScrop.txt};
	
		\addplot table [x=dt,y=L2V3] {Data/p=3 unknown/Classic/h=sqrt(tau)/averClassic_Expmult_p_3_hmax_0.01767766952966378_dtfine_0.00078125_gScrop.txt};
	\addplot table [x=dt,y=L2V3] {Data/p=3 unknown/Full/h=sqrt(tau)/averFull_Expmult_p_3_hmax_0.01767766952966378_dtfine_0.00078125_gScrop.txt};
	\addplot table [x=dt,y=L2V3] {Data/p=3 unknown/Half/h=sqrt(tau)/averHalf_Expmult_p_3_hmax_0.01767766952966378_dtfine_0.00078125_gScrop.txt};
	
\addplot[dashed,sharp plot,update limits=false] coordinates {(1e-1,8e-2) (1e-3,8e-6)};

	\legend{
	,,,,,,,,,
	{$2$},
	};
\end{axis}
\end{tikzpicture}
\caption{Convergence for $p=3$ of $\bfv_c^{\text{EM}}$~(red), $\bfv_c^{\text{Half}}$~(blue) and $\bfv_c^{\text{Full}}$~(green) towards $\bfv_f$ measured in $d_{L^\infty L^2}^{\,\text{aver}}$~(solid, left), $d_{L^\infty L^2}^{\,\text{point}}$~(dashed, left), $d_{L^2 V}^{\,\text{classic}}$~(solid, right), $d_{L^2 V}^{\,\text{inner}}$~(dashed, right) and $d_{L^2 V}^{\,\text{outer}}$~(dash dotted, right). In the top row we use $\tau \approx h$ and in the bottom row $\tau \approx h^2$.}
\label{fig:Unknown-p=3}
\end{center}
\end{figure}

\begin{figure}
\begin{center}
\begin{tikzpicture}
\begin{axis}[
clip=false,
width=.5\textwidth,
height=.45\textwidth,
xmode = log,
ymode = log,
cycle multi list={\nextlist MyColors},
xlabel={$\tau \approx h$},
scale = {1},
clip = true,
legend cell align=left,
legend style={legend columns=1,legend pos= south east,font=\fontsize{7}{5}\selectfont}
]
	\addplot table [x=dt,y=LinftyL2aver] {Data/p=1_5 unknown/Classic/h=tau/averClassic_Expmult_p_1.5_hmax_0.01767766952966378_dtfine_0.00078125_gScrop.txt};
	\addplot table [x=dt,y=LinftyL2aver] {Data/p=1_5 unknown/Full/h=tau/averFull_Expmult_p_1.5_hmax_0.01767766952966378_dtfine_0.00078125_gScrop.txt};
	\addplot table [x=dt,y=LinftyL2aver] {Data/p=1_5 unknown/Half/h=tau/averHalf_Expmult_p_1.5_hmax_0.01767766952966378_dtfine_0.00078125_gScrop.txt};
	
	\addplot table [x=dt,y=LinftyL2clas] {Data/p=1_5 unknown/Classic/h=tau/averClassic_Expmult_p_1.5_hmax_0.01767766952966378_dtfine_0.00078125_gScrop.txt};
	\addplot table [x=dt,y=LinftyL2clas] {Data/p=1_5 unknown/Full/h=tau/averFull_Expmult_p_1.5_hmax_0.01767766952966378_dtfine_0.00078125_gScrop.txt};
	\addplot table [x=dt,y=LinftyL2clas] {Data/p=1_5 unknown/Half/h=tau/averHalf_Expmult_p_1.5_hmax_0.01767766952966378_dtfine_0.00078125_gScrop.txt};
	
\addplot[dashed,sharp plot,update limits=false] coordinates {(1e-1,5e-2) (1e-3,5e-6)};

	\legend{
	,,,,,,
	{$2$},
	};
\end{axis}
\end{tikzpicture}
\begin{tikzpicture}
\begin{axis}[
clip=false,
width=.5\textwidth,
height=.45\textwidth,
xmode = log,
ymode = log,
cycle multi list={\nextlist MyColors},
xlabel={$\tau \approx h$},
scale = {1},
clip = true,
legend cell align=left,
legend style={legend columns=1,legend pos= south east,font=\fontsize{7}{5}\selectfont}
]
		\addplot table [x=dt,y=L2V1] {Data/p=1_5 unknown/Classic/h=tau/averClassic_Expmult_p_1.5_hmax_0.01767766952966378_dtfine_0.00078125_gScrop.txt};
	\addplot table [x=dt,y=L2V1] {Data/p=1_5 unknown/Full/h=tau/averFull_Expmult_p_1.5_hmax_0.01767766952966378_dtfine_0.00078125_gScrop.txt};
	\addplot table [x=dt,y=L2V1] {Data/p=1_5 unknown/Half/h=tau/averHalf_Expmult_p_1.5_hmax_0.01767766952966378_dtfine_0.00078125_gScrop.txt};
	
	\addplot table [x=dt,y=L2V2] {Data/p=1_5 unknown/Classic/h=tau/averClassic_Expmult_p_1.5_hmax_0.01767766952966378_dtfine_0.00078125_gScrop.txt};
	\addplot table [x=dt,y=L2V2] {Data/p=1_5 unknown/Full/h=tau/averFull_Expmult_p_1.5_hmax_0.01767766952966378_dtfine_0.00078125_gScrop.txt};
	\addplot table [x=dt,y=L2V2] {Data/p=1_5 unknown/Half/h=tau/averHalf_Expmult_p_1.5_hmax_0.01767766952966378_dtfine_0.00078125_gScrop.txt};
	
		\addplot table [x=dt,y=L2V3] {Data/p=1_5 unknown/Classic/h=tau/averClassic_Expmult_p_1.5_hmax_0.01767766952966378_dtfine_0.00078125_gScrop.txt};
	\addplot table [x=dt,y=L2V3] {Data/p=1_5 unknown/Full/h=tau/averFull_Expmult_p_1.5_hmax_0.01767766952966378_dtfine_0.00078125_gScrop.txt};
	\addplot table [x=dt,y=L2V3] {Data/p=1_5 unknown/Half/h=tau/averHalf_Expmult_p_1.5_hmax_0.01767766952966378_dtfine_0.00078125_gScrop.txt};
	
\addplot[dashed,sharp plot,update limits=false] coordinates {(1e-1,8e-2) (1e-3,8e-6)};

	\legend{
	,,,,,,,,,
	{$2$},
	};
\end{axis}
\end{tikzpicture}
\begin{tikzpicture}
\begin{axis}[
clip=false,
width=.5\textwidth,
height=.45\textwidth,
xmode = log,
ymode = log,
cycle multi list={\nextlist MyColors},
xlabel={$\tau \approx h^2$},
scale = {1},
clip = true,
legend cell align=left,
legend style={legend columns=1,legend pos= south east,font=\fontsize{7}{5}\selectfont}
]
	\addplot table [x=dt,y=LinftyL2aver] {Data/p=1_5 unknown/Classic/h=sqrt(tau)/averClassic_Expmult_p_1.5_hmax_0.01767766952966378_dtfine_0.00078125_gScrop.txt};
	\addplot table [x=dt,y=LinftyL2aver] {Data/p=1_5 unknown/Full/h=sqrt(tau)/averFull_Expmult_p_1.5_hmax_0.01767766952966378_dtfine_0.00078125_gScrop.txt};
	\addplot table [x=dt,y=LinftyL2aver] {Data/p=1_5 unknown/Half/h=sqrt(tau)/averHalf_Expmult_p_1.5_hmax_0.01767766952966378_dtfine_0.00078125_gScrop.txt};
	
	\addplot table [x=dt,y=LinftyL2clas] {Data/p=1_5 unknown/Classic/h=sqrt(tau)/averClassic_Expmult_p_1.5_hmax_0.01767766952966378_dtfine_0.00078125_gScrop.txt};
	\addplot table [x=dt,y=LinftyL2clas] {Data/p=1_5 unknown/Full/h=sqrt(tau)/averFull_Expmult_p_1.5_hmax_0.01767766952966378_dtfine_0.00078125_gScrop.txt};
	\addplot table [x=dt,y=LinftyL2clas] {Data/p=1_5 unknown/Half/h=sqrt(tau)/averHalf_Expmult_p_1.5_hmax_0.01767766952966378_dtfine_0.00078125_gScrop.txt};
	
\addplot[dashed,sharp plot,update limits=false] coordinates {(1e-1,5e-2) (1e-3,5e-6)};

	\legend{
	,,,,,,
	{$2$},
	};
\end{axis}
\end{tikzpicture}
\begin{tikzpicture}
\begin{axis}[
clip=false,
width=.5\textwidth,
height=.45\textwidth,
xmode = log,
ymode = log,
cycle multi list={\nextlist MyColors},
xlabel={$\tau \approx h^2$},
scale = {1},
clip = true,
legend cell align=left,
legend style={legend columns=1,legend pos= south east,font=\fontsize{7}{5}\selectfont}
]
		\addplot table [x=dt,y=L2V1] {Data/p=1_5 unknown/Classic/h=sqrt(tau)/averClassic_Expmult_p_1.5_hmax_0.01767766952966378_dtfine_0.00078125_gScrop.txt};
	\addplot table [x=dt,y=L2V1] {Data/p=1_5 unknown/Full/h=sqrt(tau)/averFull_Expmult_p_1.5_hmax_0.01767766952966378_dtfine_0.00078125_gScrop.txt};
	\addplot table [x=dt,y=L2V1] {Data/p=1_5 unknown/Half/h=sqrt(tau)/averHalf_Expmult_p_1.5_hmax_0.01767766952966378_dtfine_0.00078125_gScrop.txt};
	
	\addplot table [x=dt,y=L2V2] {Data/p=1_5 unknown/Classic/h=sqrt(tau)/averClassic_Expmult_p_1.5_hmax_0.01767766952966378_dtfine_0.00078125_gScrop.txt};
	\addplot table [x=dt,y=L2V2] {Data/p=1_5 unknown/Full/h=sqrt(tau)/averFull_Expmult_p_1.5_hmax_0.01767766952966378_dtfine_0.00078125_gScrop.txt};
	\addplot table [x=dt,y=L2V2] {Data/p=1_5 unknown/Half/h=sqrt(tau)/averHalf_Expmult_p_1.5_hmax_0.01767766952966378_dtfine_0.00078125_gScrop.txt};
	
		\addplot table [x=dt,y=L2V3] {Data/p=1_5 unknown/Classic/h=sqrt(tau)/averClassic_Expmult_p_1.5_hmax_0.01767766952966378_dtfine_0.00078125_gScrop.txt};
	\addplot table [x=dt,y=L2V3] {Data/p=1_5 unknown/Full/h=sqrt(tau)/averFull_Expmult_p_1.5_hmax_0.01767766952966378_dtfine_0.00078125_gScrop.txt};
	\addplot table [x=dt,y=L2V3] {Data/p=1_5 unknown/Half/h=sqrt(tau)/averHalf_Expmult_p_1.5_hmax_0.01767766952966378_dtfine_0.00078125_gScrop.txt};
	
\addplot[dashed,sharp plot,update limits=false] coordinates {(1e-1,8e-2) (1e-3,8e-6)};

	\legend{
	,,,,,,,,,
	{$2$},
	};
\end{axis}
\end{tikzpicture}
\caption{Convergence for $p=1.5$ of $\bfv_c^{\text{EM}}$~(red), $\bfv_c^{\text{Half}}$~(blue) and $\bfv_c^{\text{Full}}$~(green) towards $\bfv_f$ measured in $d_{L^\infty L^2}^{\,\text{aver}}$~(solid, left), $d_{L^\infty L^2}^{\,\text{point}}$~(dashed, left), $d_{L^2 V}^{\,\text{classic}}$~(solid, right), $d_{L^2 V}^{\,\text{inner}}$~(dashed, right) and $d_{L^2 V}^{\,\text{outer}}$~(dash dotted, right). In the top row we use $\tau \approx h$ and in the bottom row $\tau \approx h^2$.}
\label{fig:Unknown-p=1_5}
\end{center}
\end{figure}
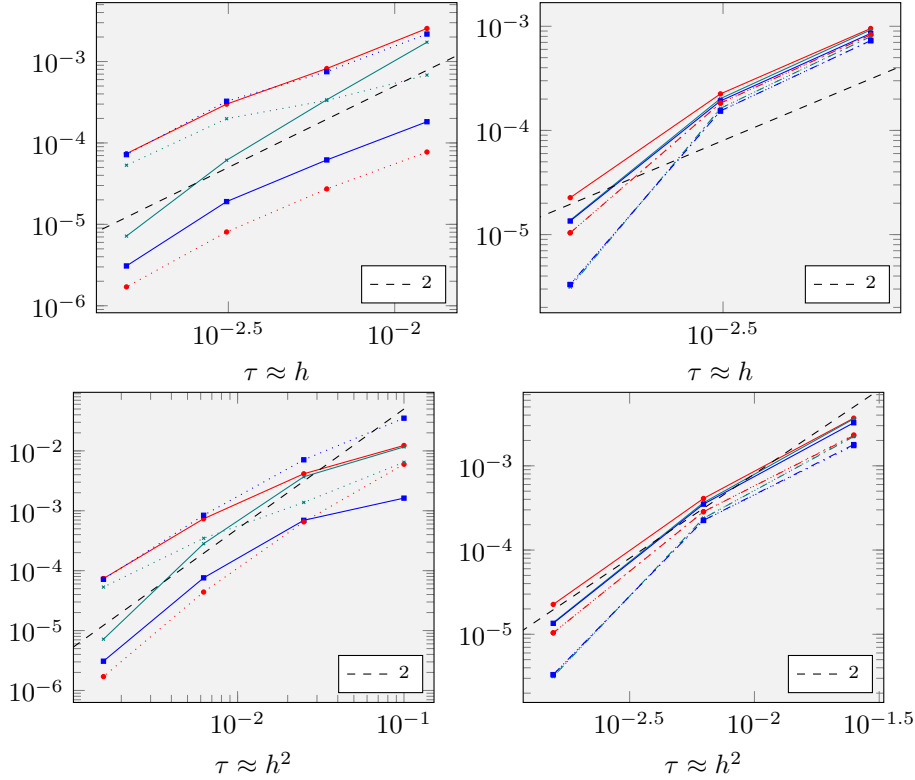

\subsection{Conclusion}
Experimentally, there is no evidence for the divergence of the classical Euler-Maruyama scheme~\eqref{eq:algo-EM}. Still convergence might fail. Our algorithms~\eqref{eq:algo} and~\eqref{algo:2nd} are safe and simple substitutions. They achieve optimal linear convergence in space and optimal $1/2$-convergence in time with minimal regularity assumptions. Additionally, the algorithms~\eqref{eq:algo},~\eqref{algo:2nd} and~\eqref{eq:algo-EM} share the same computational complexity, since solving the nonlinear system in each time step is the most expensive operation. The construction of the random inputs is fast and can be done with the simple sampling algorithm~\eqref{algo:Sampling}.

\printbibliography 


\end{document}